\documentclass[12pt,leqno]{article}
\usepackage{hyperref}
\usepackage{soul}
\usepackage[doc]{optional}
\usepackage{tikz}
\usepackage{xcolor}
\definecolor{labelkey}{rgb}{0,0.08,0.45}
\definecolor{rekey}{rgb}{0,0.6,0.0}
\definecolor{Brown}{rgb}{0.45,0.0,0.05}
\usepackage{exscale,relsize}
\usepackage{amsmath}
\usepackage{amsfonts}
\usepackage{amssymb}
\usepackage{calc}
\usepackage{theorem}
\usepackage{pifont}      
\usepackage{graphicx}
\DeclareMathOperator{\weakstarly}{\stackrel{\mathrm{w*}}{\rightharpoondown}}

\newcommand{\wk}{\ensuremath{\operatorname{w*}}}

\newcommand{\scal}[2]{\langle{{#1},{#2}}\rangle}

\newcommand{\RR}{\ensuremath{\mathbb R}}

\newcommand{\RX}{\ensuremath{\,\left]-\infty,+\infty\right]}}
\newcommand{\RXX}{\ensuremath{\,\left[-\infty,+\infty\right]}}

\newcommand{\NN}{\ensuremath{\mathbb N}}

\newcommand{\thalb}{\ensuremath{\tfrac{1}{2}}}

\newcommand{\menge}[2]{\big\{{#1} \mid {#2}\big\}}

\newcommand{\To}{\ensuremath{\rightrightarrows}}

\newcommand{\spand}{\operatorname{span}}

\newcommand{\dom}{\ensuremath{\operatorname{dom}}}

\newcommand{\gra}{\ensuremath{\operatorname{gra}}}

\newcommand{\inte}{\ensuremath{\operatorname{int}}}

\newcommand{\ran}{\ensuremath{\operatorname{ran}}}

\newcommand{\conv}{\ensuremath{\operatorname{conv}}}

\renewcommand{\phi}{\ensuremath{\varphi}}

\newcommand{\J}{\ensuremath{\mathbf{J}}}
\newcommand{\qede}{\hspace*{\fill}$\Diamond$\medskip}

\newtheorem{theorem}{Theorem}[section]
\newtheorem{lemma}[theorem]{Lemma}
\newtheorem{fact}[theorem]{Fact}
\newtheorem{corollary}[theorem]{Corollary}
\newtheorem{proposition}[theorem]{Proposition}
\newtheorem{definition}[theorem]{Definition}

\theoremstyle{plain}{\theorembodyfont{\rmfamily}
}
\theoremstyle{plain}{\theorembodyfont{\rmfamily}
}
\theoremstyle{plain}{\theorembodyfont{\rmfamily}
}
\theoremstyle{plain}{\theorembodyfont{\rmfamily}
\newtheorem{example}[theorem]{Example}}
\theoremstyle{plain}{\theorembodyfont{\rmfamily}
\newtheorem{remark}[theorem]{Remark}}

\theoremstyle{plain}{\theorembodyfont{\rmfamily}
}

\begin{document}


\title{\sffamily{Construction of  pathological maximally monotone operators on non-reflexive Banach spaces
 }}

\author{
Heinz H.\ Bauschke\thanks{Mathematics, Irving K.\ Barber School,
University of British Columbia, Kelowna, B.C. V1V 1V7, Canada.
E-mail: \texttt{heinz.bauschke@ubc.ca}.},\; Jonathan M.
Borwein\thanks{CARMA, University of Newcastle, Newcastle, New South
Wales 2308, Australia. E-mail:
\texttt{jonathan.borwein@newcastle.edu.au}.  Distinguished Professor
King Abdulaziz University, Jeddah. },\;
 Xianfu
Wang\thanks{Mathematics, Irving K.\ Barber School, University of British Columbia,
Kelowna, B.C. V1V 1V7, Canada. E-mail:
\texttt{shawn.wang@ubc.ca}.},\;
and Liangjin\
Yao\thanks{Mathematics, Irving K.\ Barber School, University of British Columbia,
Kelowna, B.C. V1V 1V7, Canada.
E-mail:  \texttt{ljinyao@interchange.ubc.ca}.}}

\date{August 6,   2011}
\maketitle

\begin{abstract} \noindent
In this paper,  we construct maximally monotone
 operators that are not of Gossez's dense-type (D) in many nonreflexive
 spaces. Many of these operators also fail to possess the
 Br{\o}nsted-Rockafellar (BR) property.
  Using  these operators,
we show that the partial inf-convolution of two BC--functions will
not always be a BC--function. This provides a negative answer to a
challenging question posed by Stephen Simons. Among other consequences, we
deduce that every Banach space which contains an isomorphic copy of
the James space $\J$ or its dual $\J^*$,  or  $c_0$ or
its dual  $\ell^1$, admits a non type (D) operator.
\end{abstract}

\noindent


\noindent

\section{Preliminaries}

Throughout this paper, we assume that $X$ is a real Banach space
with norm $\|\cdot\|$, that $X^*$ is the continuous dual of $X$, and
that $X$ and $X^*$ are paired by $\scal{\cdot}{\cdot}$. As usual, we
identify $X$ with its canonical image in the bidual space $X^{**}$.
Furthermore, $X\times X^*$ and $(X\times X^*)^*: = X^*\times X^{**}$
are likewise paired via $\scal{(x,x^*)}{(y^*,y^{**})}:=
\scal{x}{y^*} + \scal{x^*}{y^{**}}$, where $(x,x^*)\in X\times X^*$
and $(y^*,y^{**}) \in X^*\times X^{**}$.

 Let $A\colon
X\To X^*$ be a \emph{set-valued operator} (also known as a
multifunction) from $X$ to $X^*$, i.e., for every $x\in X$,
$Ax\subseteq X^*$, and let $\gra A:= \menge{(x,x^*)\in X\times
X^*}{x^*\in Ax}$ be the \emph{graph} of $A$.  The \emph{domain} of
$A$ is $\dom A:= \menge{x\in X}{Ax\neq\varnothing}$,  and $\ran
A:=A(X)$ for the \emph{range} of $A$. Recall that $A$ is
\emph{monotone} if
\begin{equation}
\scal{x-y}{x^*-y^*}\geq 0,\quad \forall (x,x^*)\in \gra A\;
\forall (y,y^*)\in\gra A,
\end{equation}
and \emph{maximally monotone} if $A$ is monotone and $A$ has
 no proper monotone extension
(in the sense of graph inclusion).
Let $A:X\rightrightarrows X^*$ be monotone and $(x,x^*)\in X\times X^*$.
 We say $(x,x^*)$ is \emph{monotonically related to}
$\gra A$ if
\begin{align*}
\langle x-y,x^*-y^*\rangle\geq0,\quad \forall (y,y^*)\in\gra
A.\end{align*}

We now recall the three fundamental subclasses of maximally
monotone operators.

 \begin{definition}\label{def1}
 Let $A:X\To X^*$ be maximally monotone.
 Then three key types of monotone operators are defined as follows.
 \begin{enumerate}
 \item $A$ is
\emph{of dense type or type (D)} (1971, \cite{Gossez3}  and \cite{ph2}) if for every
$(x^{**},x^*)\in X^{**}\times X^*$ with
\begin{align*}
\inf_{(a,a^*)\in\gra A}\langle a-x^{**}, a^*-x^*\rangle\geq 0,
\end{align*}
there exist a  bounded net
$(a_{\alpha}, a^*_{\alpha})_{\alpha\in\Gamma}$ in $\gra A$
such that
$(a_{\alpha}, a^*_{\alpha})_{\alpha\in\Gamma}$
weak*$\times$strong converges to
$(x^{**},x^*)$.
\item $A$ is
\emph{of type negative infimum (NI)} (1996, \cite{SiNI}) if
\begin{align*}
\sup_{(a,a^*)\in\gra A}\big(\langle a,x^*\rangle+\langle a^*,x^{**}\rangle
-\langle a,a^*\rangle\big)
\geq\langle x^{**},x^*\rangle,
\quad \forall(x^{**},x^*)\in X^{**}\times X^*.
\end{align*}

\item $A$ is
\emph{of  ``Br{\o}nsted-Rockafellar" (BR) type  (1999, \cite{Si6})
if whenever $(x,x^*)\in X\times X^*$, $\alpha,\beta>0$ while
\begin{align*}\inf_{(a,a^*)\in\gra A} \langle x-a,x^*-a^*\rangle
>-\alpha\beta\end{align*} then there exists $(b,b^*)\in\gra A$ such
that $\|x-b\|<\alpha,\|x^*-b^*\|<\beta$.}
\end{enumerate}
\end{definition}
As we shall see below in Fact \ref{PF:Su1}, it is now known that the
first two classes coincide. This coincidence is central to many of
our proofs. Fact~\ref{MAS:BR1}
 also shows us that every maximally monotone operator of type (D) is of
type (BR)(The converse fails, see
Example~\ref{FPEX:1}\ref{BCCE:A7}.).
 Moreover, in reflexive space every maximally monotone operator is
of type (D), as is the subdifferential operator of every closed convex function
on a Banach space. While monotone operator theory is rather complete
in reflexive space
--- and for type (D) operators in general space --- the general
situation is less clear \cite{BorVan,Bor3}. Hence our continuing
interest in operators which are not of type (D).

We shall say  a Banach space $X$ is \emph{of type (D)} \cite{Bor3}
if every maximally monotone operator  on $X$ is of  type (D). At
present the only known type (D) spaces are the reflexive spaces; and
our work here suggests that there are no non-reflexive type (D)
spaces.  In \cite[Exercise 9.6.3]{BorVan} such spaces were called
(NI) spaces and some potential non-reflexive examples were
conjectured; all of which are ruled out by our current work. In
\cite[Theorem 9.79]{BorVan} a variety of the pleasant properties of
type (D) spaces was listed.

\subsection{More preliminary technicalities}

 Maximal monotone operators have proven to be a potent class of
objects in modern Optimization and Analysis; see, e.g.,
\cite{Bor1,Bor2,Bor3}, the books \cite{BC2011,
BorVan,BurIus,ph,Si,Si2,RockWets,Zalinescu} and the references
therein.

We adopt standard notation used in these books especially
\cite[Chapter 2]{BorVan} and \cite{Bor1, Si, Si2}: Given a subset
$C$ of $X$,
the \emph{indicator function} of $C$, written as $\iota_C$, is defined
at $x\in X$ by
\begin{align}
\iota_C (x):=\begin{cases}0,\,&\text{if $x\in C$;}\\
+\infty,\,&\text{otherwise}.\end{cases}\end{align}
 The \emph{closed unit
ball} is $B_X:=\menge{x\in X}{\|x\|\leq 1}$, and
$\NN:=\{1,2,3,\ldots\}$.

Let $\alpha,\beta\in\RR$. In the sequel it will also be useful to
let $\delta_{\alpha,\beta}$ be defined by $\delta_{\alpha,\beta}:=1$,
if $\alpha=\beta$; $\delta_{\alpha,\beta}:=0$, otherwise.

For a subset $C^*$ of $X^*$, $\overline{C^*}^{\wk}$
  is the weak$^{*}$ closure of $C^*$.
If $Z$ is a real  Banach space with dual $Z^*$ and a set $S\subseteq
Z$, we denote $S^\bot$ by $S^\bot := \{z^*\in Z^*\mid\langle
z^*,s\rangle= 0,\quad \forall s\in S\}$. Given a subset $D$ of
$Z^*$, we define $D_{\bot}$ \cite{PheSim} by $D_\bot := \{z\in
Z\mid\langle z,d^*\rangle= 0,\quad \forall d^*\in D\}$.

The \emph{adjoint} of an operator  $A$, written $A^*$, is defined by
\begin{equation*}
\gra A^* :=
\menge{(x^{**},x^*)\in X^{**}\times X^*}{(x^*,-x^{**})\in(\gra A)^{\bot}}.
\end{equation*}
We say $A$ is a \emph{linear relation} if $\gra A$ is a linear subspace.
 We say that $A$ is
\emph{skew} if $\gra A \subseteq \gra (-A^*)$;
equivalently, if $\langle x,x^*\rangle=0,\; \forall (x,x^*)\in\gra A$.
Furthermore,
$A$ is \emph{symmetric} if $\gra A
\subseteq\gra A^*$; equivalently, if $\scal{x}{y^*}=\scal{y}{x^*}$,
$\forall (x,x^*),(y,y^*)\in\gra A$.
We define the \emph{symmetric part} and the \emph{skew part} of $A$ via
\begin{equation}
\label{Fee:1}
P := \thalb A + \thalb A^* \quad\text{and}\quad
S:= \thalb A - \thalb A^*,
\end{equation}
respectively. It is easy to check that $P$ is symmetric and that $S$
is skew. Let $A:X\rightrightarrows X^*$ be monotone and $S$ be a
subspace of $X$.
 We say $A$ is \emph{$S$--saturated} \cite{Si2} if
\begin{align*}
Ax+S^{\bot}=Ax,\quad
\forall x\in\dom A.
\end{align*}
We say a maximally monotone operator $A:X\rightrightarrows X^*$ is
\emph{unique} if all maximally monotone extensions of
$A$ (in the sense of graph inclusion) in $X^{**}\times X^*$ coincide.

 Let $f\colon X\to \RX$. Then
$\dom f:= f^{-1}(\RR)$ is the \emph{domain} of $f$, and $f^*\colon
X^*\to\RXX\colon x^*\mapsto \sup_{x\in X}(\scal{x}{x^*}-f(x))$ is
the \emph{Fenchel conjugate} of $f$. We say $f$ is proper if $\dom f\neq\varnothing$.
Let $f$ be proper. The \emph{subdifferential} of
$f$ is defined by
   $$\partial f\colon X\To X^*\colon
   x\mapsto \{x^*\in X^*\mid(\forall y\in
X)\; \scal{y-x}{x^*} + f(x)\leq f(y)\}.$$
 For $\varepsilon \geq 0$,
the \emph{$\varepsilon$--subdifferential} of $f$ is defined by
   $$\partial_{\varepsilon} f\colon X\To X^*\colon
   x\mapsto \menge{x^*\in X^*}{(\forall y\in
X)\; \scal{y-x}{x^*} + f(x)\leq f(y)+\varepsilon}.$$ Note that
$\partial f= \partial_{0}f$.
 We  denote  by $J:=J_X$ the duality map, i.e.,
the subdifferential of the function $\tfrac{1}{2}\|\cdot\|^2$
mapping $X$ to $X^*$.

 Now let $F:X\times X^*\rightarrow\RX$.
 We say $F$ is a \emph{BC--function} (BC stands for
``Bigger conjugate'') \cite{Si2} if $F$ is proper and
convex with
\begin{align} F^*(x^*,x)
 \geq F(x,x^*)\geq\langle x,x^*\rangle\quad\forall(x,x^*)\in X\times X^*.
 \end{align}

Let $Y$ be another real Banach space. We set  $P_X: X\times Y\rightarrow
X\colon (x,y)\mapsto x$,
 and
 $P_Y: X\times Y\rightarrow Y\colon (x,y)\mapsto y$.
Let $L:X\rightarrow Y$ be linear. We say $L$ is a (linear)
\emph{isomorphism} into $Y$ if $L$ is one to one, continuous and
$L^{-1}$ is continuous on $\ran L$. We say $L$ is an \emph{isometry}
if $\|Lx\|=\|x\|,  \forall x\in X$. The spaces $X$, $Y$ are then
\emph{isometric} (\emph{isomorphic}) if there exists an isometry
(\emph{isomorphism})  from $X$ onto $Y$.

Let $F_1, F_2\colon X\times Y\rightarrow\RX$.
Then the \emph{partial inf-convolution} $F_1\Box_1 F_2$
is the function defined on $X\times Y$ by
\begin{equation*}F_1\Box_1 F_2\colon
(x,y)\mapsto \inf_{u\in X}
F_1(u,y)+F_2(x-u,y).
\end{equation*}
Then $F_1\Box_2 F_2$
is the function defined on $X\times Y$ by
\begin{equation*}F_1\Box_2 F_2\colon
(x,y)\mapsto \inf_{v\in Y}
F_1(x,y-v)+F_2(x,v).
\end{equation*}
In Example~\ref{FPEX:1}\ref{BCCE:A2}\&\ref{BCCE:A4} of  this paper,
 we provide a negative answer to
the following question  posed by S.\ Simons
\cite[Problem~22.12]{Si2}:
\begin{quote} \emph{Let $F_1,F_2:X\times X^*\rightarrow\RX$ be
 proper lower semicontinuous and convex. Assume that
$F_1, F_2$ are BC--functions and  that
\begin{align*}\text{$\bigcup_{\lambda>0} \lambda\left[P_{X^*}
\dom F_1-P_{X^*}\dom F_2\right]$ is a closed subspace of $X^*$.}
\end{align*} Is $F_1\Box_1 F_2$ necessarily a BC--function?}\end{quote}

We are now ready to set to work. The paper is organized as follows.
In Section~\ref{s:aux}, we collect auxiliary results for future
reference and for the reader's convenience. Our main result
(Theorem~\ref{PBABD:2}) is established in Section~\ref{s:main}. In
Section~\ref{EAPT},  we provide various applications and extensions
including the promised negative answer to Simons' question.
Furthermore, we show that every Banach space containing an
isomorphic copy of the James space $\textbf{J}$ or of  $\textbf{J}^*$,
 of $\ell^{1}$ or of $c_0$ is
not of type (D) (Example~\ref{FPEX:1}\ref{BCCE:A06} or
Corollary~\ref{DSChE:2}, Corollary~\ref{DSChE:1} and  Example~\ref{FPEX:C1}).

 \section{Auxiliary results}\label{s:aux}
Observation:

\begin{fact}\emph{(See \cite[Proposition~2.6.6(c)]{Megg}).}\label{Megg:1}
Let  $D$ be a subspace of $X^*$. Then
$(D_{\bot})^{\bot}=\overline{D}^{\wk}$.
\end{fact}

We now record a famous  Banach space  result:

\begin{fact}[Banach and Mazur]
\emph{(See {\cite[Theorem 5.8, page~240]{FabianHH2}
or \cite[Theorem 5.17, page~144]{FabianHH})}.)}
\label{isom:1a}
Every separable Banach space
 is isometric to a subspace of  $C[0,1]$.
 \end{fact}

Now we turn to  prerequisite results on Fitzpatrick functions,
monotone operators, and  linear relations.

\begin{fact}[Fitzpatrick]
\emph{(See {\cite[Corollary~3.9 and Proposition~4.2]{Fitz88}} and
\cite{Bor1,BorVan}.)} \label{f:Fitz} Let $A\colon X\To X^*$ be
maximally monotone, and set
\begin{equation}
F_A\colon X\times X^*\to\RX\colon
(x,x^*)\mapsto \sup_{(a,a^*)\in\gra A}
\big(\scal{x}{a^*}+\scal{a}{x^*}-\scal{a}{a^*}\big),
\end{equation}
which is the \emph{Fitzpatrick function} associated with $A$.
Then $F_A$ is a BC--function and $F_A=\langle\cdot,\cdot\rangle$ on $\gra A$.
\end{fact}

\begin{fact}[Simons and Z\u{a}linescu]
\emph{(See \cite[Theorem~4.2]{SiZ} or \cite[Theorem~16.4(a)]{Si2}.)}\label{F4}
Let $Y$ be a real Banach space and $F_1, F_2\colon X\times Y \to \RX$ be proper,
lower semicontinuous, and convex. Assume that
for every $(x,y)\in X\times Y$,
\begin{equation*}(F_1\Box_2 F_2)(x,y)>-\infty
\end{equation*}
and that  $\bigcup_{\lambda>0} \lambda\left[P_X\dom F_1-P_X\dom F_2\right]$
is a closed subspace of $X$. Then for every $(x^*,y^*)\in X^*\times Y^*$,
\begin{equation*}
(F_1\Box_2 F_2)^*(x^*,y^*)=\min_{u^*\in X^*}
\left[F_1^*(x^*-u^*,y^*)+F_2^*(u^*,y^*)\right].
\end{equation*}\end{fact}

\begin{fact}[Simons  and Z\u{a}linescu]\emph{(See \cite[Theorem~16.4(b)]{Si2}.)}\label{BCCFA:1}
Let $Y$ be a real Banach space and $F_1, F_2\colon X\times Y \to \RX$ be proper,
lower semicontinuous and convex. Assume that
for every $(x,y)\in X\times Y$,
\begin{equation*}(F_1\Box_1 F_2)(x,y)>-\infty
\end{equation*}
and that  $\bigcup_{\lambda>0} \lambda\left[P_Y\dom F_1-P_Y\dom F_2\right]$
is a closed subspace of $Y$. Then for every $(x^*,y^*)\in X^*\times Y^*$,
\begin{equation*}
(F_1\Box_1 F_2)^*(x^*,y^*)=\min_{v^*\in Y^*}
\left[F_1^*(x^*,v^*)+F_2^*(x^*,y^*-v^*)\right].
\end{equation*}\end{fact}

Phelps and Simons proved the next Fact~\ref{PF:1} for unbounded
 linear operators in \cite[Proposition~3.2(a)]{PheSim}, but their
proof can also be adapted for general linear relations.
 For reader's convenience, we write down  their proof.

\begin{fact}[Phelps and Simons]
\label{PF:1} Let $A: X\rightrightarrows X^*$ be a monotone linear relation.
 Then $(x,x^*)\in X\times
X^*$ is monotonically related to $\gra A$ if and only if
\begin{equation*}
\langle x,x^*\rangle\geq 0\; \text{and}\;
 \left[\langle y^*,x\rangle+\langle x^*,y\rangle\right]^2
\leq 4\langle x^*,x\rangle\langle y^*,y\rangle,\quad \forall(y,y^*)\in\gra A.
\end{equation*}

\end{fact}
\begin{proof}
We have the following equivalences:
\begin{align*}
& (x,x^*)\in X\times
X^*\  \text{is monotonically related to $\gra A$}\\
&\Leftrightarrow \lambda^2\langle y,y^*\rangle-\lambda\left[
\langle y^*,x\rangle+\langle x^*,y\rangle\right]+\langle x,x^*\rangle=\langle
 \lambda y^*-x^*,\lambda y-x\rangle\geq0,\forall \lambda\in\RR, \forall(y,y^*)\in\gra A\\
&\Leftrightarrow\langle x,x^*\rangle\geq 0\; \text{and}\;
 \left[\langle y^*,x\rangle+\langle x^*,y\rangle\right]^2
\leq 4\langle x^*,x\rangle\langle y^*,y\rangle,\ \forall(y,y^*)\in\gra A\
\text{(by \cite[Lemma~2.1]{PheSim})}.
\end{align*}
This completes the proof.
\end{proof}

\begin{fact}[Simons / Marques Alves and Svaiter]\emph{(See
\cite[Lemma~15]{SiNI} or \cite[Theorem~36.3(a)]{Si2}, and
\cite[Theorem~4.4]{MarSva}.)} \label{PF:Su1} Let $A:X \To X^*$ be
maximally  monotone. Then $A$ is of type (D) if and only if it is
of type (NI).
\end{fact}

 We next cite some properties regarding the \emph{uniqueness} of
 (maximally) monotone  extension of a maximally monotone operator to $X^{**} \times X^*$.
 Simons  showed that every maximally monotone operator of type (NI) is \emph{unique} in \cite{Si4}.
 Recently, Marques Alves and Svaiter contributed the following results:
 \begin{fact}[Marques Alves and Svaiter]\emph{(See
\cite[Theorem~1.6]{MarSva2}.)}
\label{MAS:UN1} Let $A:X \To X^*$ be a maximally  monotone linear relation
 that is not of type (D). Assume that $A$ is unique.
Then $\gra A=\dom F_A$.
\end{fact}

 \begin{fact}[Marques Alves and Svaiter]\emph{(See
\cite[Corollary~4.6]{MarSva}.)} \label{MAS:UNc2} Let $A:X \To X^*$
be a maximally  monotone operator such that $\gra A$ is not affine.
Then $A$ is of type (D) if and only if $A$ is unique.
\end{fact}

The Gossez operator defined as in Example~\ref{FPEX:1}\ref{BCCE:A6}
is a maximally monotone and unique operator that is not of type (D)
\cite{Gossez1}.

 The definition of operators of type (BR) directly yields
the following result.
\begin{fact}\label{BRFa:1}
Let $A:X\rightrightarrows X^*$ be maximally monotone
 and $(x,x^*)\in X\times X^*$. Assume that $A$
is of type (BR) and that
$\inf_{(a,a^*)\in\gra A}\langle x-a,x^*-a^*\rangle>-\infty$. Then
$x\in\overline{\dom A}$ and $x^*\in\overline{\ran A}$.
\end{fact}

Additionally,

 \begin{fact}[Marques Alves and Svaiter]\emph{(See
\cite[Theorem~1.4(4)]{MarSva2} or \cite{MarSva3}.)} \label{MAS:BR1} Let $A:X \To X^*$
be a maximally  monotone operator. Assume that $A$ is of type (NI).
Then $A$ is of type (BR).
\end{fact}

We shall also need some precise results about linear relations. The
 first two are elementary.

\begin{fact}[Cross]\label{Rea:1}
\emph{(See \cite[Proposition~I.2.8(a)]{Cross}.)} Let $A \colon X \To
Y$ be a linear relation. Then $(\forall (x,x^*)\in\gra A)$ $Ax=x^*
+A0$.
\end{fact}

\begin{lemma}
\label{PF:A1} Let $A:X\rightrightarrows X^*$ be a linear relation.
Assume that $A^*$ is monotone. Then $\ker A^*\subseteq(\ran
A^*)^{\bot}$.
\end{lemma}
\begin{proof}
Let $x^{**}\in\ker A^*$ and then $(\alpha x^{**},0)\in\gra A^*,
\forall \alpha\in\RR$. Then
\begin{align*}0&\leq\langle \alpha x^{**}+y^{**}, y^*\rangle
=\alpha\langle x^{**}, y^{*}\rangle+\langle y^{**}, y^*\rangle,
\quad \forall (y^{**},y^*)\in\gra A^*,
\forall\alpha\in\RR.\end{align*} Hence $\langle x^{**},
y^*\rangle=0,\quad
  \forall (y^{**},y^*)\in\gra A^*$ and thus $x^{**}\in(\ran A^*)^{\bot}$.
Thus $\ker A^*\subseteq(\ran A^*)^{\bot}$.
\end{proof}

\begin{fact}\emph{(See \cite[Theorem~3.1]{BBWY:1}.)}\label{TypeDe:1}
Let $A:X\rightrightarrows X^*$ be a maximally monotone linear relation.  Then
$A$ is of type (D)
if and only if $A^*$ is monotone.
\end{fact}

\begin{fact}\emph{(See \cite[Theorem~3.1]{Yao2}.)}
\label{lisum:1}Let $A:X\To X^*$ be a maximally monotone linear relation,
and let $f:X\rightarrow \RX$ be a proper lower semicontinuous convex function
with $\dom A\cap\inte\dom \partial f\neq\varnothing$.  Then
$A+\partial f$ is maximally monotone.\end{fact}

\begin{fact}[Simons]\label{Satu:1}\emph{(See \cite[Theorem~28.9]{Si2}.)}
 Let $Y$ be a Banach space, and $L:Y\rightarrow X$ be continuous and linear with $\ran L$ closed and
  $\ran L^*=Y^*$.
 Let $A:X\rightrightarrows X^*$ be  monotone
 with $\dom A\subseteq\ran L$ such that $\gra A\neq\varnothing$.
  Then $A$ is maximally monotone
 if, and only if $A$ is $\ran L$--saturated and $L^*AL$ is maximally monotone.
\end{fact}

\begin{theorem}\label{Simonco:1}
 Let $Y$ be a Banach space, and $L:Y\rightarrow X$ be an isomorphism into $X$.
 Let $T:Y\rightrightarrows Y^*$ be monotone.
  Then $T$ is maximally monotone
 if, and only if  $(L^*)^{-1}TL^{-1}$, mapping $X$ into $X^*$, is maximally monotone.
\end{theorem}
\begin{proof}
Let $A=(L^*)^{-1}TL^{-1}$. Then $\dom A\subseteq\ran L$.
  Since  $L$ is an  isomorphism into $X$, $\ran L$ is closed.
By \cite[Theorem~3.1.22(b)]{Megg} or \cite[Exercise~2.39(i), page~59]{FabianHH}, $\ran L^*=Y^*$.
Hence $\gra(L^*)^{-1}TL^{-1}\neq\varnothing$ if and only if $\gra T\neq\varnothing$.
Clearly, $A$ is monotone.
Since $\{0\}\times(\ran L)^{\bot}\subseteq\gra (L^*)^{-1}$ and then by Fact~\ref{Rea:1},
 $A=(L^*)^{-1}TL^{-1}$ is $\ran L$--saturated.
By Fact~\ref{Satu:1},
$A=(L^*)^{-1}TL^{-1}$ is maximally monotone
if and only if  $L^*AL=T$ is maximally monotone.
\end{proof}

The following consequence will allow us to construct maximally
monotone operators that are not of type (D) in a variety of
non-reflexive Banach spaces.

\begin{corollary}[Subspaces]\label{Simonco:2}
 Let $Y$ be a Banach space, and $L:Y\rightarrow X$ be an isomorphism into $X$.
 Let $T:Y\rightrightarrows Y^*$ be maximally monotone. Assume that $T$ is not of type (D).
Then  $(L^*)^{-1}TL^{-1}$ is a maximally monotone operator mapping
$X$ into $X^*$ that
 is not of type (D). In particular, every Banach subspace of a type (D)
 space is of type (D).
\end{corollary}
\begin{proof}
By Theorem~\ref{Simonco:1}, $(L^*)^{-1}TL^{-1}$ is maximally
monotone. By Fact~\ref{PF:Su1}, there exists $(y^{**}_0,y^*_0)\in
Y^{**}\times Y^*$ such that \allowdisplaybreaks
\begin{align}
\sup_{(b,b^*)\in\gra T}\big\{\langle y^{**}_0,b^*\rangle
+\langle y^*_0,b\rangle-\langle b,b^*\rangle\big\}<\langle y^{**}_0,y^*_0\rangle.\label{Brefs:1}
\end{align}
By \cite[Theorem~3.1.22(b)]{Megg} or  \cite[Exercise~2.39(i), page~59]{FabianHH},
$\ran L^*=Y^*$ and thus
 there exists $x^*_0\in X^*$ such that $L^*x^*_0=y^*_0$.
Let $A=(L^*)^{-1}TL^{-1}$.
Then we have
\begin{align}
&\sup_{(a,a^*)\in\gra A}\big\{\langle L^{**}y^{**}_0,a^*\rangle
+\langle x^*_0,a\rangle-\langle a,a^*\rangle\big\}\nonumber\\
&=\sup_{(Ly,a^*)\in\gra A}\big\{\langle y^{**}_0, L^*a^*\rangle
+\langle x^*_0,Ly\rangle-\langle Ly,a^*\rangle\big\}\nonumber\\
&=\sup_{(Ly,a^*)\in\gra A}\big\{\langle y^{**}_0, L^*a^*\rangle
+\langle L^*x^*_0,y\rangle-\langle y,L^*a^*\rangle\big\}\nonumber\\
&=\sup_{(Ly,a^*)\in\gra A}\big\{\langle y^{**}_0, L^*a^*\rangle
+\langle y^*_0,y\rangle-\langle y,L^*a^*\rangle\big\}\nonumber\\
&=\sup_{(y,y^*)\in\gra T}\big\{\langle y^{**}_0, y^*\rangle
+\langle y^*_0,y\rangle-\langle y,y^*\rangle\big\}
\quad \text{(by $(Ly,a^*)\in\gra A\Leftrightarrow (y,L^*a^*)\in\gra T$)} \nonumber\\
&<\langle y^{**}_0,y^*_0\rangle\quad \text{(by \eqref{Brefs:1})}\nonumber\\
&=\langle L^{**}y^{**}_0, x^*_0\rangle.
\end{align}
Thus $A$ is not of type (NI) and
 hence $A=(L^*)^{-1}TL^{-1}$ is not of type (D) by Fact~\ref{PF:Su1}.
\end{proof}

Note that it follows that $X$ is of type (D) whenever $X^{**}$ is.

\section{Main result}\label{s:main}

We start with several technical tools. To relate Fitzpatrick
functions and skew operators we have:

\begin{lemma}\label{LeSK:a1}
Let $A\colon X\To X^*$ be a skew linear relation. Then
\begin{align}
F_A=\iota_{\gra (-A^*)\cap X\times X^*}.\label{Lesk:1}
\end{align}
\end{lemma}
\begin{proof}
Let $(x_0,x_0^*)\in X\times X^*$. We have
\begin{align}
F_A(x_0,x_0^*)&=\sup_{(x,x^*)\in\gra A}\{
\langle (x^*_0,x_0),(x,x^*)\rangle-\langle x,x^*\rangle\}\nonumber\\
&=\sup_{(x,x^*)\in\gra A}
\langle (x^*_0,x_0),(x,x^*)\rangle\nonumber\\
&=\iota_{(\gra A)^{\bot}}(x^*_0,x_0)\nonumber\\
&=\iota_{\gra(-A^*)}(x_0,x^*_0)\nonumber\\
&=\iota_{\gra (-A^*)\cap X\times X^*} (x_0,x^*_0).\nonumber\end{align}
Hence \eqref{Lesk:1} holds.
\end{proof}

To produce operators not of type (D) but that are of (BR) we
exploit:

\begin{lemma}\label{MSCBR:L1}
Let $A:X\rightrightarrows X^*$ be a maximally monotone and  linear skew operator.
Assume that $\gra(-A^*)\cap X\times X^*\subseteq\gra A$.
Then $A$ is of type (BR).
\end{lemma}
\begin{proof}
Let $\alpha,\beta>0$ and $(x,x^*)\in X\times X^*$ be such that $\inf_{(a,a^*)\in\gra A}
\langle x-a,x^*-a^*\rangle >-\alpha\beta$. Since $A$ is skew, we have
\begin{align}
\inf_{(a,a^*)\in\gra A}\langle x,x^*\rangle-\left[\langle x,a^*\rangle+\langle a,x^*\rangle\right]=
\inf_{(a,a^*)\in\gra A}\langle x-a,x^*-a^*\rangle >-\alpha\beta.
\end{align}
Thus, $\langle x,a^*\rangle+\langle a,x^*\rangle=0, \forall
(a,a^*)\in\gra A$ and hence $(x,x^*)\in\gra(-A^*)$. Then by
assumption, $(x,x^*)\in\gra A$.  Taking $(b,b^*)=(x,x^*)$, we have
$\|b-x\|<\alpha$ and $\|b^*-x^*\|<\beta$. Hence $A$ is of type (BR).
\end{proof}

\begin{corollary}\label{MSCUni:1}
Let $A:X\rightrightarrows X^*$ be a maximally monotone and  linear
skew operator that  is not of type (D). Assume that $A$ is unique.
Then $\gra A=\gra(-A^*)\cap X\times X^*$   and so $A$ is of type
(BR).
\end{corollary}

\begin{proof}
Apply Fact~\ref{MAS:UN1}, Lemma~\ref{LeSK:a1} and Lemma~\ref{MSCBR:L1} directly.
\end{proof}

\begin{proposition}\label{ProJon}
Let $A \colon X \rightrightarrows X^*$ be  maximally monotone.
Assume that $A$ is of type (NI) and  that there exists $e\in X^*$
such that
\begin{align*} \langle x^*,x \rangle \ge \langle e,x \rangle^2,\quad \forall
(x,x^*) \in \gra A.
\end{align*}
Then $e \in \overline{\conv\ran A}$.
\end{proposition}

\begin{proof}  Suppose $e  \not \in  \overline{\conv\ran A}$.
 Then by the Separation Theorem, there exists
 $x_0^{**} \in X^{**}$ such that $\langle e-x^*,x_0^{**}\rangle\ge 1$ for all $x^*
\in \ran A$.  Then we have
\begin{align*}\langle x^*-e,x-x_0^{**}\rangle &=
 \langle e-x^*,x_0^{**}\rangle +\langle x^*-e,x\rangle,\quad \forall(x,x^*)\in\gra A
\\&\ge 1+\langle e,x\rangle^2-\langle e,x\rangle
\\&\ge  \min _{t\in\RR} t^2-t+1 =\frac{3}{4}.
\end{align*}
Thus $A$ is not of type (NI), which contradicts the assumption.
\end{proof}

The proof of the following result was partially inspired by that \cite[Proposition~2.2]{BuSv}.
\begin{proposition}\label{ProJon1}
Let $A \colon X \rightrightarrows X^*$ be  a maximally monotone linear relation. Assume that
there exists $e\in X^*$
such that
$e \notin \overline{\ran A}$ and
 that
\begin{align*} \langle x^*,x \rangle \ge \langle e,x \rangle^2,\quad \forall
(x,x^*) \in \gra A.
\end{align*}
Then  $A$ is neither of type (D) nor unique.
\end{proposition}

\begin{proof}  By Proposition~\ref{ProJon}, $A$ is not of type (NI) and hence $A$ is not of
type (D) by Fact~\ref{PF:Su1}.
Similar to the proof of Proposition~\ref{ProJon}, there exists
 $x_0^{**} \in X^{**}$ such that $\langle e,x_0^{**}\rangle\ge 1$ and $x^{**}_0\in(\ran A)^{\bot}$.
Let $0<\alpha<2$. Then we have
\begin{align*}\langle x^*-\alpha e,x-\tfrac{1}{\alpha}x_0^{**}\rangle &=
 \langle \alpha e-x^*,\tfrac{1}{\alpha}x_0^{**}\rangle
  +\langle x^*-\alpha e,x\rangle,\quad \forall(x,x^*)\in\gra A
\\&\ge 1+\langle e,x\rangle^2-\alpha\langle e,x\rangle
\\&\ge  \min _{t\in\RR} t^2-\alpha t+1\\
& =1-\frac{\alpha^2}{4}>0.
\end{align*}
Thus  for every $0<\alpha<2$, $(\tfrac{1}{\alpha}x^{**}_0, \alpha e)\in X^{**}\times X^*$
 is monotonically related to $\gra A$.
 Take $0<\alpha_1<\alpha_2<2$. Then by  Zorn's Lemma,
  we have a maximally monotone extension,
 $A_1:X^{**}\rightrightarrows X^*$ such that $\gra A_1\supseteq\gra A\cup\{(\tfrac{1}{\alpha_1}x^{**}_0,\alpha_1 e,)\}$,
 and we can also obtain a maximally monotone extension, $A_2 : X^{**}\rightrightarrows X^*$ such that
 $\gra A_2 \supseteq\gra A\cup \{(\tfrac{1}{\alpha_2} x_0^{**},\alpha_2 e)\}$.

 Now we show $\gra A_1\neq\gra A_2$. Suppose to the contrary that $\gra A_1=\gra A_2$. Then by the monotonicity of $A_1$,
   we have
\begin{align}\langle \tfrac{1}{\alpha_1}x^{**}_0-\tfrac{1}{\alpha_2}x^{**}_0,
 \alpha_1 e-\alpha_2 e\rangle\geq 0.\label{nonUn:p1}
 \end{align}
 On the other hand,
 \begin{align*}
 \langle \tfrac{1}{\alpha_1}x^{**}_0-\tfrac{1}{\alpha_2}x^{**}_0,
 \alpha_1 e-\alpha_2 e\rangle&=(\alpha_1-\alpha_2)(\tfrac{1}{\alpha_1}-\tfrac{1}{\alpha_2})\langle x^{**}_0,e\rangle\\
 &<(\alpha_1-\alpha_2)(\tfrac{1}{\alpha_1}-\tfrac{1}{\alpha_2})<0,
 \end{align*}
which contradicts \eqref{nonUn:p1}. Hence $\gra A_1\neq\gra A_2$ and thus $A$ is not unique.
\end{proof}

We are now ready to establish our work-horse Theorem~\ref{PBABD:2},
which allows  us to  construct various maximally monotone operators
--- both linear and nonlinear --- that are not of type (D). The idea of constructing the operators in
the following
 fashion is based upon  \cite[Theorem~5.1]{BB} and was stimulated by \cite{BuSv}.
\begin{theorem}[Predual constructions]
\label{PBABD:2}
Let $A: X^*\rightarrow X^{**}$ be linear and continuous.
 Assume that $\ran A \subseteq X$ and that there exists $e\in X^{**}\backslash X$ such that
\begin{align*}
\langle Ax^*,x^*\rangle=\langle e,x^*\rangle^2,\quad \forall x^*\in X^*.
\end{align*}
Let $ P$ and $S$ respectively  be the symmetric part and antisymmetric
 part of $A$.  Let $T:X\rightrightarrows X^*$  be defined by
\begin{align}\gra T&:=\big\{(-Sx^*,x^*)\mid x^*\in X^*, \langle e, x^*\rangle=0\big\}\nonumber\\
&=\big\{(-Ax^*,x^*)\mid x^*\in X^*, \langle e, x^*\rangle=0\big\}.\label{PBABA:a1}
\end{align}
Let $f:X\rightarrow\RX$ be a proper lower semicontinuous and convex function.
 Set $F:=f\oplus f^*$ on $ X\times X^*$.
Then the following hold.
\begin{enumerate}
\item\label{PBAB:em01}
$A$ is a maximally monotone operator on $X^*$ that is neither of type (D) nor unique.
\item\label{PBAB:emmaz1}
$Px^*=\langle x^*,e\rangle e,\ \forall x^*\in X^*.$

\item\label{PBAB:em1}
 $T$
is maximally monotone and skew on $X$.

\item\label{PBAB:emma1}
$\gra  T^*=\{(Sx^*+re,x^*)\mid x^*\in X^*,\ r\in\RR\}$.

\item\label{PBAB:emma2}
$-T$ is not maximally monotone.

\item\label{PBAB:emm1}
 $T$
is not of type (D).

\item\label{PBAB:em2}
$F_T=\iota_C$, where
\begin{align}
C:=\{(-Ax^*,x^*)\mid x^*\in X^*\}.
 \end{align}

 \item\label{PBAB:emu2}
$T$ is not unique.

\item\label{PBAB:emr3}
$T$ is not of type (BR).

 \item\label{BCC:0a2}   If $\dom T\cap\inte\dom\partial f\neq\varnothing$,
then $T+\partial f$ is maximally monotone.
\item \label{BCC:02}$F$ and $F_T$ are BC--functions  on $X\times X^*$.
\item\label{BCC:2}Moreover,  \begin{align*}\bigcup_{\lambda>0} \lambda
\big(P_{X^*}(\dom F_T)-P_{X^*}(\dom F)\big)=X^*,\end{align*} while,
assuming that there exists $(v_0,v_0^*)\in X\times X^*$ such that
\begin{align}
f^*(v_0^*)+f^{**}(v_0-A^*v^*_0)<\langle
v_0,v^*_0\rangle\label{IeSp:3},
\end{align} \index{BC--function} then
$F_T\Box_1F$ is not a BC--function.

 \item\label{BCC:3} Assume that
$\left[\ran A-\bigcup_{\lambda>0} \lambda\dom f\right]$
is a closed subspace of $X$ and that
$$\varnothing\neq\dom f^{**}\small\circ A^*|_{X^*}\nsubseteqq \{e\}_{\bot}.$$
Then $T+\partial f$ is not of type (D).

\item\label{BCC:4}
Assume that $\dom f^{**}=X^{**}$.
Then $T+\partial f$ is a maximally monotone operator that is not of type (D).
\end{enumerate}
\end{theorem}
\begin{proof}
\ref{PBAB:em01}: Clearly, $A$ has full domain. Since $A$ is monotone and continuous, $A$
is maximally monotone. By the assumptions that
$e\notin X$ and $\overline{\ran A} \subseteq \overline{X}=X$,
then by Proposition~\ref{ProJon1}, $A$ is neither of type (D) nor unique.
See also \cite[Theorem~14.2.1 and Theorem~13.2.3]{BAU:1} for alternative proof of that $A$ is
not of type (D).

\ref{PBAB:emmaz1}: Now we show that
\begin{align}Px^*=\langle x^*,e\rangle e,\ \forall x^*\in X^*.\label{SuSy:1}
\end{align}
Since $\langle \cdot,e\rangle e=\partial (\tfrac{1}{2}\langle \cdot,e\rangle^2)$
 and by  \cite[Theorem~5.1]{PheSim},
$\langle \cdot,e\rangle e$ is a symmetric operator on $X^*$.
 Clearly, $A-\langle \cdot,e\rangle e$ is skew.
Then \eqref{SuSy:1} holds.

\ref{PBAB:em1}:
Let $x^*\in X^*$ with $ \langle e, x^*\rangle=0$. Then we have
\begin{align*}
Sx^*=\langle x^*,e\rangle e+Sx^*=Px^*+Sx^*=Ax^*\in \ran A\subseteq X.
\end{align*}
Thus \eqref{PBABA:a1} holds and $T$ is well defined.

We have $S$ is skew and hence $T$ is skew.
 Let $(z,z^*)\in X\times X^*$ be monotonically related to $\gra T$.
By Fact~\ref{PF:1}, we have \begin{align}
0&=\langle z, x^*\rangle+\langle -Sx^*,z^*\rangle
=\langle z+Sz^*, x^*\rangle,\quad \forall x^*\in \{e\}_{\bot}.\nonumber
\end{align}
Thus by Fact~\ref{Megg:1}, we have
$ z+Sz^*\in  (\{e\}_{\bot})^{\bot}=\spand\{e\}$ and then
\begin{align}z=-Sz^*+\kappa e,\  \exists \kappa\in\RR.\label{BAB:9}
\end{align}
By $(0,0)\in\gra T$,
\begin{align} \kappa\langle z^*,e\rangle=\langle -Sz^*
+\kappa e, z^*\rangle=\langle z, z^*\rangle\geq0.\label{BAB:10}
\end{align}
Then by \eqref{BAB:9} and \ref{PBAB:emmaz1},  \begin{align}
Az^*=Pz^*+Sz^*=Pz^*+\kappa e-z=\left[\langle z^*,e\rangle+\kappa\right]e-z.\label{BAB:010}
\end{align}
By the assumptions that $z\in X$, $Az^*\in X$ and $e\notin X$,
 $\left[\langle z^*,e\rangle+\kappa\right]=0$ by \eqref{BAB:010}. Then by \eqref{BAB:10}, we have
 $\langle  z^*,e\rangle=\kappa=0$ and thus $(z,z^*)\in\gra T$ by \eqref{BAB:9}.
  Hence $T$ is maximally monotone.

\ref{PBAB:emma1}: Let $(x^{**}_0,x^*_0)\in X^{**}\times X^*$. Then we have
\begin{align*}
&(x^{**}_0,x^*_0)\in \gra T^*\Leftrightarrow \langle x^*_0,
 Sx^*\rangle+\langle x^*, x^{**}_0\rangle=0, \quad \forall x^*\in \{e\}_{\bot}\\
&\Leftrightarrow \langle x^*, x^{**}_0-Sx^*_0\rangle=0, \quad \forall x^*\in \{e\}_{\bot}\\
&\Leftrightarrow  x^{**}_0-Sx^*_0\in (\{e\}_{\bot})^{\bot}
=\spand\{e\}\quad\text{(by Fact~\ref{Megg:1})}\\
&\Leftrightarrow  x^{**}_0-Sx^*_0=re,\quad \exists r\in\RR.
\end{align*}
Thus
$
\gra  T^*=\{(Sx^*+re,x^*)\mid x^*\in X^*,\ r\in\RR\}$.

\ref{PBAB:emma2}:
Since $e\notin X$, we have $e\neq0$. Then there exists $z^*\in X^*$ such that
$z^*\not\in\{e\}_{\bot}$. Then by \ref{PBAB:emmaz1}\&\ref{PBAB:emma1}
 and the assumption that $\ran A\subseteq X$,
 we have \begin{align*}
(Az^*, z^*)=(Sz^*+\langle e, z^*\rangle e, x^*)\in\gra T^*\cap X\times X^*.
\end{align*}
Thus we have
\begin{align*}
\langle Az^*-x, z^*-x^*\rangle &=\langle Az^*, z^*\rangle-
\left[\langle Az^*,x^*\rangle+\langle x,z^*\rangle\right]+\langle x,x^*\rangle\\
 &=\langle Az^*, z^*\rangle\geq0,\quad \forall (x,x^*)\in\gra(-T).
\end{align*}
Hence $(Az^*,z^*)$ is monotonically related to $\gra (-T)$. Since $z^*\notin\ran (-T)$,
 $(Az^*,z^*)\notin\gra (-T)$ and then $-T$ is not maximally monotone.

\ref{PBAB:emm1}:
By \ref{PBAB:emma1}, $T^*$ is not monotone. Then by Fact~\ref{TypeDe:1}, $T$ is not of type (D).

\ref{PBAB:em2}:
By \ref{PBAB:emma1}, we have \begin{align*}
& (z,z^*)\in\gra  (-T^*)\cap X\times X^*\\&\Leftrightarrow
(z,z^*)=(-Sz^*-re,z^*),\quad z\in X,\  \exists r\in\RR, \ z^*\in X^*\\
 &\Leftrightarrow (z,z^*)=(-Sz^*-\langle z^*,e\rangle e+
 \left[\langle z^*,e\rangle -r\right]e,z^*),     \quad z\in X,\ \exists r\in\RR,\  z^*\in X^*\\
  &\Leftrightarrow (z,z^*)=(-A z^*+\left[\langle z^*,e\rangle -r\right]e,z^*),
       \quad z\in X,\ \exists r\in\RR,\  z^*\in X^*\ \text{(by \ref{PBAB:emmaz1})}\\
   &\Leftrightarrow (z,z^*)=(-A z^*,z^*), \quad  \langle z^*,e\rangle =r\
    \text{(since $z, Az^*\in X$ and $e\notin X$)}, \exists r\in\RR,\  z^*\in X^*\\
  &\Leftrightarrow (z,z^*)
 \in\{(-Ax^* ,x^*)\mid x^*\in X^*\}=C. \nonumber
 \end{align*}
Thus by Lemma~\ref{LeSK:a1}, we have
$F_T=\iota_C$.

\ref{PBAB:emu2}:
Since $e\notin X$, we have $e\neq0$. Then there exists $z^*\in X^*$ such that
$z^*\not\in\{e\}_{\bot}$. Thus $z^*\notin\ran T$. By \ref{PBAB:em2},
 $z^*\in P_{X^*}\left[\dom F_T\right]$. Thus, $\gra T\neq\dom F_T$.  Then by
\ref{PBAB:emm1} and Fact~\ref{MAS:UN1}, $T$ is not unique.

\ref{PBAB:emr3}: Suppose to the contrary that $T$ is of type (BR).
Let $z^*$ be as in the proof of \ref{PBAB:emu2}. Then by Lemma~\ref{LeSK:a1} and \ref{PBAB:em2},
 we have $(-Az^*,z^*)\in\gra(-T^*)\cap X\times X^*$ and then
\begin{align*}\inf_{(a,a^*)\in\gra T}\langle
-Az^*-a, z^*-a^*\rangle=\langle
-Az^*, z^*\rangle>-\infty.
\end{align*}
Then Fact~\ref{BRFa:1} shows $z^*\in\overline{\ran T}$, which
contradicts that
 $z^*\notin\{e\}_{\bot}=\overline{\ran T}$.
 Hence $T$ is not of type (BR).

\ref{BCC:0a2}: Apply \ref{PBAB:em1} and  Fact~\ref{lisum:1}.

\ref{BCC:02}:
Clearly, $F$ is a BC--function. By \ref{PBAB:em1} and Fact~\ref{f:Fitz},
we see that $F_T$ is a BC--function.

\ref{BCC:2}:
By \ref{PBAB:em2}, we have
\begin{align}\bigcup_{\lambda>0} \lambda
\big(P_{X^*}(\dom F_T)-P_{X^*}(\dom F)\big)=X^*.\label{Sppl:a01}\end{align}
Then for every $(x,x^*)\in X\times X^*$ and $u\in X$, by \ref{BCC:02},
\begin{equation*}
F_T (x-u,x^*)+ F(u,x^*)=F_T (x-u,x^*)+ (f\oplus f^*)(u,x^*)\geq \langle x-u,x^*\rangle+\langle
u,x^*\rangle=\langle x,x^*\rangle.
\end{equation*}
Hence
\begin{align}(F_T\Box_1F)(x,x^*)\geq\langle x,x^*\rangle>-\infty.\label{Sppl:a1}
\end{align}
Then by \eqref{Sppl:a01}, \eqref{Sppl:a1} and Fact~\ref{BCCFA:1},
\begin{align}
(F_T\Box_1F)^* (v^*_0,v_0)
&=
\min_{x^{**}\in X^{**}}F^*_T ( v^*_0,x^{**})+ F^*( v^*_0,v_0-x^{**})\nonumber\\
&\leq F^*_T ( v^*_0,A^*v^*_0)+ F^*( v^*_0,v_0-A^*v^*_0)\nonumber\\
&=0+  F^*( v^*_0,v_0-A^*v^*_0)\quad\text{(by \ref{PBAB:em2})} \nonumber\\
&=(f\oplus f^*)^*( v^*_0,v_0-A^*v^*_0)
=(f^*\oplus f^{**})( v^*_0,v_0-A^*v^*_0)\nonumber\\
&=f^*(v^*_0)+f^{**}(v_0-A^*v^*_0)\nonumber\\
&<\langle  v^*_0, v_0\rangle\quad\text{(by \eqref{IeSp:3})}. \nonumber
\end{align}
Hence $F_T\Box_1 F$ is not a BC--function.

\ref{BCC:3}:  By the assumption, there exists
$x_0^*\in\dom f^{**}\small\circ A^*|_{X^*}$ such that $\langle e, x_0^*\rangle\neq0$.
Let $\varepsilon_0=\frac{\langle e,x_0^*\rangle^2}{2}$.
By  \cite[Theorem~2.4.4(iii)]{Zalinescu}),
there exists $y^{***}_0\in \partial_{\varepsilon_0}f^{**}(A^*x^*_0)$.
By  \cite[Theorem~2.4.2(ii)]{Zalinescu}),
\begin{align}
f^{**}(A^* x_0^*)+f^{***}(y^{***}_0)\leq\langle A^* x_0^*,y_0^{***}\rangle +\varepsilon_0.
\end{align}
\allowdisplaybreaks
Then by \cite[Lemma~45.9]{Si2} or
the proof of \cite[Eq.(2.5) in Proposition~1]{Rock702}, there exists $y^*_0\in X^*$ such that
\begin{align}
f^{**}(A^* x_0^*)+f^*(y_0^*)<\langle A^*x_0^*, y^*_0\rangle+2\varepsilon_0.\label{IeSp:s1}
\end{align}
Let $z_0^{*}=y_0^{*}+x^*_0$.
  Then by \eqref{IeSp:s1}, we have\begin{align}
f^{**}(A^* x_0^*)+f^*(z_0^*-x^*_0)&
<\langle A^*x_0^*, z^*_0-x^*_0\rangle+2\varepsilon_0\nonumber\\
&=\langle A^*x_0^*, z^*_0\rangle-
\langle A^*x_0^*, x^*_0\rangle+2\varepsilon_0\nonumber\\
&=\langle A^*x_0^*, z^*_0\rangle-
\langle x_0^*, Ax^*_0\rangle+2\varepsilon_0\nonumber\\
&=\langle A^*x_0^*, z^*_0\rangle-2\varepsilon_0+2\varepsilon_0\nonumber\\
&=\langle A^*x_0^*, z^*_0\rangle\label{IeSp:s3}.
\end{align}
Then for every $(x,x^*)\in X\times X^*$ and $u^*\in X$, by \ref{BCC:02},
\begin{equation*}
F_T (x,x^*-u^*)+ F(x,u^*)=F_T (x,x^*-u^*)+
 (f\oplus f^*)(x,u^*)\geq \langle x,x^*-u^*\rangle+\langle
x,u^*\rangle=\langle x,x^*\rangle.
\end{equation*}
Hence
\begin{align}(F_T\Box_2F)(x,x^*)\geq\langle x,x^*\rangle>-\infty.\label{Sppl:sa1}
\end{align}
Then by \eqref{Sppl:sa1}, \ref{PBAB:em2} and Fact~\ref{F4},
\begin{align}
(F_T\Box_2F)^* (z^*_0,A^*x^*_0)
&=
\min_{y^{*}\in X^{*}}F^*_T ( y^*,A^*x^*_0)+ F^*( z^*_0-y^*,A^*x^*_0)\nonumber\\
&\leq F^*_T ( x^*_0,A^*x^*_0)+ F^*(z^*_0-x_0^*,A^*x^*_0)\nonumber\\
&=0+  F^*( z_0^*-x_0^*,A^*x^*_0)\quad\text{(by \ref{PBAB:em2})} \nonumber\\
&=(f\oplus f^*)^*( z^*_0-x_0^*,A^*x^*_0)\nonumber\\
&=f^*(z^*_0-x_0^*)+f^{**}(A^*x^*_0)\nonumber\\
&<\langle z^*_0, A^*x^*_0\rangle\quad\text{(by \eqref{IeSp:s3})}.\label{TBCE:s1}
\end{align}
Let $F_0:X\times X^*\rightarrow\RX$ be defined by
\begin{align}
(x,x^*)\mapsto\langle x,x^*\rangle+\iota_{\gra (T+\partial f)}(x,x^*).
\end{align}
Clearly,
$F_T\Box_2F\leq F_0$ on $X\times X^*$ and thus $(F_T\Box_2F)^*\geq F^*_0$ on $X^*\times X^{**}$.
By \eqref{TBCE:s1},
$F^*_0(z^*_0, A^*x^*_0)<\langle z^*_0, A^*x^*_0\rangle$. Hence $T+\partial f$
is not of type (NI) and thus $T+\partial f$ is not of type (D) by Fact~\ref{PF:Su1}.

\ref{BCC:4}:
 Since $\dom f^{**}=X^{**}$, $\dom f=X$ by \cite[Theorem~2.3.3]{Zalinescu}.
By $\dom f^{**}=X^{**}$ again, $\dom f^{**}
\small\circ A_{\alpha}^*|_{X^*}=X^*\nsubseteqq \{\alpha\}_{\bot}$.
Then apply \ref{BCC:0a2}\&\ref{BCC:3} directly.
\end{proof}

\begin{remark}[Grothendieck spaces \cite{BorVan}] In light of part (xiii) of the previous theorem),  we record  that for a closed convex function $$\dom f = X \mbox{~implies~}  \dom f^{**}=X^{**} \Leftrightarrow X \mbox{~is a \emph{Grothendieck} space.}$$ All reflexive spaces are  Grothendieck spaces while all non-reflexive Grothendieck spaces (such as $L^\infty[0,1]$) contain an isomorphic copy of $c_0$.\qede \end{remark}
We are now ready to exploit Theorem \ref{PBABD:2}.

\section{Examples and applications}\label{EAPT}

We begin with the case of $c_0$ and its dual $\ell^1$.

\subsection{Applications to $c_0$}\label{sec:c}

\begin{example}[$c_0$]\label{FPEX:1}
 Let $ \text{$X: = c_0$, with norm $\|\cdot\|_{\infty}$ so that
  $X^* = \ell^1$ with norm $\|\cdot\|_{1}$,}
$ and  $X^{**}=\ell^{\infty}$  with its second dual norm
$\|\cdot\|_{*}$. Let
$\alpha:=(\alpha_n)_{n\in\NN}\in\ell^{\infty}$ with $\limsup
\alpha_n\neq0$, and let
$A_{\alpha}:\ell^1\rightarrow\ell^{\infty}$ be defined  by
\begin{align}\label{def:Aa}
(A_{\alpha}x^*)_n:=\alpha^2_nx^*_n+2\sum_{i>n}\alpha_n \alpha_ix^*_i,
\quad \forall x^*=(x^*_n)_{n\in\NN}\in\ell^1.\end{align}
\allowdisplaybreaks Now let $ P_{\alpha}$ and $S_{\alpha}$
respectively
  be the symmetric part and antisymmetric
 part of $A_{\alpha}$.  Let $T_{\alpha}:c_{0}\rightrightarrows X^*$  be defined by
\begin{align}\gra T_{\alpha}&
:=\big\{(-S_{\alpha} x^*,x^*)\mid x^*\in X^*,
 \langle \alpha, x^*\rangle=0\big\}\nonumber\\
&=\big\{(-A_{\alpha} x^*,x^*)\mid x^*\in X^*,
 \langle \alpha, x^*\rangle=0\big\}\nonumber\\
&=\big\{\big((-\sum_{i>n}
\alpha_n \alpha_ix^*_i+\sum_{i<n}\alpha_n \alpha_ix^*_i)_n, x^*\big)
\mid x^*\in X^*, \langle \alpha, x^*\rangle=0\big\}.\label{PBABA:Ea1}
\end{align}

Then
\begin{enumerate}
\item\label{BCCE:A01} $\langle A_{\alpha}x^*,x^*\rangle=\langle \alpha , x^*\rangle^2,
\quad \forall x^*=(x^*_n)_{n\in\NN}\in\ell^1$
and\eqref{PBABA:Ea1} is well defined.

\item \label{BCCE:SA01} $A_{\alpha}$ is a maximally monotone
 operator on $\ell^1$ that is neither of type (D) nor unique.

\item\label{BCCE:A1} $T_{\alpha}$
is a maximally monotone  operator on $c_0$ that is not of type (D).

\item\label{BCCE:Ac1} $-T_{\alpha}$
is  not  maximally monotone.

\item \label{BCCE:Au2} $T_{\alpha}$ is neither unique nor of type (BR).

\item\label{BCCE:A2} $F_{T_{\alpha}}\Box_1
 (\|\cdot\|\oplus\iota_{B_{X^*}})$ is not a BC--function.\index{BC--function}
\item  \label{BCCE:A3}  $T_{\alpha}+
\partial \|\cdot\|$ is a maximally monotone operator on $c_0(\NN)$ that is not of type (D).
\item\label{BCCE:A4} If $\tfrac{1}{\sqrt{2}}<\|\alpha\|_*\leq 1$, then
$F_{T_{\alpha}}\Box_1 (\tfrac{1}{2}\|\cdot\|^2\oplus \tfrac{1}{2}\|\cdot\|^2_{1})$
 is not a BC--function.
\item \label{BCCE:A5}  For  $\lambda>0$,  $T_{\alpha}+\lambda J$
 is a maximally monotone operator on $c_0$ that is not of type (D).
\item\label{BCCE:A5a}  Let $\lambda>0$ and
 a linear isometry $L$
 mapping $c_0$ to  a subspace of $C[0,1]$ be given. Then
both  $(L^*)^{-1}(T_{\alpha}+\partial \|\cdot\|)L^{-1}$ and
 $(L^*)^{-1}(T_{\alpha}+\lambda J)L^{-1}$ are
  maximally monotone operators that are  not of type (D).
Hence neither $c_0$ nor $C[0,1]$ is of type (D).
\item \label{BCCE:A06} Every Banach space that contains an isomorphic copy of $c_0$ is not of type (D).

\item \label{BCCE:A6}  Let $G:\ell^1\rightarrow\ell^{\infty} $
 be Gossez's operator  \cite{Gossez1} defined by
\begin{align*}
\big(G(x^*)\big)_n:=\sum_{i>n}x^*_i-
\sum_{i<n}x^*_i,\quad \forall(x^*_n)_{n\in\NN}\in\ell^1.
\end{align*}
Then $T_e: c_0\To\ell^1$ as defined by
\begin{align*}
\gra T_e:=\{(-G(x^*),x^*)\mid x^*\in\ell^1, \langle x^*, e\rangle=0\}
\end{align*}
is a maximally monotone operator that is not of type (D), where
$e:=(1,1,\ldots,1,\ldots)$.

\item \label{BCCE:A7} Moreover, $G$ is a  unique maximally monotone operator that is not of type (D),
but $G$ is of type (BR).

\end{enumerate}
\end{example}
\begin{proof}
We have $\alpha\notin c_0$.
Since $\alpha=(\alpha_n)_{n\in\NN}\in\ell^{\infty}$ and $\|A_{\alpha}\|\leq2\|\alpha\|^2$, $A_{\alpha}$ is
 linear and continuous
  and $\ran A_{\alpha}\subseteq c_0\subseteq \ell^{\infty}$.

 \ref{BCCE:A01}:
 We have
\begin{align}
\langle A_{\alpha}x^*,x^*\rangle&
=\sum_{n}x^*_n(\alpha^2_n x^*_n+2\sum_{i>n}\alpha_n \alpha_ix^*_i)\nonumber\\
&=
\sum_{n}\alpha_n^2{x^*_n}^2
+2\sum_{n}\sum_{i>n}\alpha_n\alpha_i x^*_n x^*_i\nonumber\\
&=\sum_{n}\alpha^2_n{x^*_n}^2+\sum_{n\neq i}\alpha_n\alpha_i x^*_n x^*_i\nonumber\\
&=(\sum_{n}\alpha_n {x^*_n})^2
=\langle \alpha ,x^*\rangle^2,\quad \forall x^*=(x^*_n)_{n\in\NN}\in\ell^1.
\end{align}
Then Theorem~\ref{PBABD:2}\ref{PBAB:emmaz1}
 shows that the symmetric part  $P_{\alpha}$ of $A_{\alpha}$ is
 $P_{\alpha}x^*=\langle \alpha, x^*\rangle \alpha$ (for every $x^*\in\ell^1$).
 Thus, the skew part $S_{\alpha}$ of $A_{\alpha}$ is
\begin{align}
(S_{\alpha}x^*)_n&=(A_{\alpha}x^*)_n-(P_{\alpha}x^*)_n\nonumber\\
&=\alpha^2_nx^*_n+
2\sum_{i>n}\alpha_n \alpha_ix^*_i-\sum_{i\geq 1}\alpha_n\alpha_i x^*_i\nonumber\\
&=\sum_{i>n}\alpha_n \alpha_ix^*_i-\sum_{i<n}\alpha_n \alpha_ix^*_i.\label{ExSupm:1}
\end{align}
Then by Theorem~\ref{PBABD:2}, \eqref{PBABA:Ea1} is well defined.

\ref{BCCE:SA01}: Apply
\ref{BCCE:A01} and Theorem~\ref{PBABD:2}\ref{PBAB:em01} directly.

\ref{BCCE:A1}: Combine Theorem~\ref{PBABD:2}\ref{PBAB:em1}\&\ref{PBAB:emm1}.

\ref{BCCE:Ac1}: Apply Theorem~\ref{PBABD:2}\ref{PBAB:emma2} directly.

\ref{BCCE:Au2}: Apply
Theorem~\ref{PBABD:2}\ref{PBAB:emu2}\&\ref{PBAB:emr3}.

\ref{BCCE:A2}
 Since $\alpha\neq0$,
 there exists $i_0\in\NN$ such that $\alpha_{i_0}\neq0$.
Let $e_{i_0}:=(0,\ldots,0,1,0,\ldots)$, i.e., the $i_0$th is $1$ and the others are $0$.
 Then by \eqref{ExSupm:1}, we have
\begin{align}
S_{\alpha}e_{i_0}=\alpha_{i_0}(\alpha_1,\ldots,
\alpha_{i_0-1},0,-\alpha_{i_0+1},-\alpha_{i_0+2},\ldots).
\end{align}
Then
\begin{align}A^*e_{i_0}&=P_{\alpha}e_{i_0}-S_{\alpha}e_{i_0}\nonumber\\
&=\alpha_{i_0}(0,\ldots,0,\alpha_{i_0},2\alpha_{i_0+1}, 2\alpha_{i_0+2},
\ldots).\label{ExSupm:2}
\end{align}
Now set $v^*_0:=e_{i_0}$ and $v_0:=3\|\alpha\|^2_* e_{i_0}$. Thus by
\eqref{ExSupm:2},
\begin{align}
v_0-A^*v^*_0&= 3\|\alpha\|^2_*e_{i_0}-A^*e_{i_0}\nonumber\\
&=(0,\ldots,0,3\|\alpha\|^2_*-\alpha^2_{i_0},-2\alpha_{i_0}\alpha_{i_0+1},
 -2\alpha_{i_0}\alpha_{i_0+2},\ldots).
\label{ExSupm:3}
\end{align}
Let $f:=\|\cdot\|$ on $X=c_0$. Then
$f^*=\iota_{B_{X^*}}$ by \cite[Corollary~2.4.16]{Zalinescu}.
We have
\begin{align*}
f^*(v^*_0)+f^{**}(v_0-A^*e_{i_0})&=\iota_{B_{X^*}}(e_{i_0})+\|v_0-A^*e_{i_0}\|_* \\
&=\lVert{3\|\alpha\|_*e_{i_0}-A^*e_{i_0}}\rVert_* \\
&<3\|\alpha\|^2_*\quad \text{(by \eqref{ExSupm:3})}\\
&=\langle v_0,v^*_0\rangle.
\end{align*}
Hence by Theorem~\ref{PBABD:2}\ref{BCC:2},
$F_{T_{\alpha}}\Box_1 (\|\cdot\|\oplus\iota_{B_{X^*}})$ is not a BC--function.

\ref{BCCE:A3}:  Let $f:=\|\cdot\|$ on $X$.
Since $\dom f^{**}=X^{**}$.
Then apply Theorem~\ref{PBABD:2}\ref{BCC:4}.

\ref{BCCE:A4}:
 By $\tfrac{1}{\sqrt{2}}<\|\alpha\|_*\leq 1$,
take $|\alpha_{i_0}|^2>\tfrac{1}{2}$.
Let $e_{i_0}$ be defined as
in the proof of \ref{BCCE:A2}.
Then
take $v^*_1:=\frac{1}{2} e_{i_0}$
and $v_1:=\big(1+\frac{1}{2}\alpha^2_{i_0}\big) e_{i_0}$.

By \eqref{ExSupm:2}, we have
\begin{align}
&v_1-A^*v^*_1
=(0,\ldots,0,1,-\alpha_{i_0}\alpha_{i_0+1},-\alpha_{i_0}\alpha_{i_0+2}, \ldots).
\end{align}
Since $\left| \alpha_{i_0}\alpha_{j}\right|\leq\|\alpha\|^2_*\leq1,
\ \forall j\in\NN$, then
\begin{align}
\| v_1-A^*v^*_1\|_*\leq 1.\label{ExSupm:5}\end{align}
 Let  $f:=\tfrac{1}{2}\|\cdot\|^2$ on $X=c_0$. Then
$f^*=\tfrac{1}{2}\|\cdot\|^2_{1}$ and $f^{**}=\tfrac{1}{2}\|\cdot\|_*^2$.
\allowdisplaybreaks
We have
\begin{align*}
f^*(v^*_1)+f^{**}(v_1-A^*v^*_1)&=\tfrac{1}{2}\|v^*_1\|^2_{1}
+ \tfrac{1}{2}\| v_1-A^*v^*_1\|^2_*\\
&\leq\tfrac{1}{8}+ \tfrac{1}{2}\quad\text{(by\  \eqref{ExSupm:5})}\\
&<\tfrac{\alpha^2_{i_0}}{4}+ \tfrac{1}{2}\quad
\text{(since $\alpha^2_{i_0}>1/2$)}\\
&=\langle v^*_1,v_1\rangle.
\end{align*}
Hence by Theorem~\ref{PBABD:2}\ref{BCC:2},
$F_{T_{\alpha}}\Box_1
 (\tfrac{1}{2}\|\cdot\|^2\oplus\tfrac{1}{2}\|\cdot\|^2_*)$ is not a BC--function.

\ref{BCCE:A5}: Let $\lambda>0$ and
$f:=\tfrac{\lambda}{2}\|\cdot\|^2$ on $X=c_0$. Then
$f^{**}=\tfrac{\lambda}{2}\|\cdot\|^2_*$.
 Then apply Theorem~\ref{PBABD:2}\ref{BCC:4}.

\ref{BCCE:A5a}: Since $c_0$ is separable
by \cite[Example~1.12.6]{Megg} or \cite[Proposition~1.26(ii)]{FabianHH},
by Fact~\ref{isom:1a}, there exists a linear operator $L: c_0\rightarrow C[0,1]$
 that is an isometry from $c_0$ to  a subspace of $C[0,1]$.
Then combine \ref{BCCE:A3}\&\ref{BCCE:A5}
and Corollary~\ref{Simonco:2}.

\ref{BCCE:A06} Combine \ref{BCCE:A1} (or \ref{BCCE:A3} or \ref{BCCE:A5})
 and Corollary~\ref{Simonco:2}.

\ref{BCCE:A6}: To obtain the result on $T_e$, directly apply
\ref{BCCE:A1} (or see \cite[Example~5.2]{BB}).

\ref{BCCE:A7}  Now $-G$ is type (D) but $G$ is not \cite{BB}. To see
that $G$ is unique, note that $-G^*$ is monotone by Fact
\ref{TypeDe:1} and so provides the unique maximal extension. Since
$G$ is skew and continuous, clearly, $-G^*x^*=Gx^*, \forall
x^*\in\ell^1$. Then Lemma~\ref{MSCBR:L1} implies that $G$ is of
type (BR).  The uniqueness of $G$ was also verified in \cite[Example~14.2.2]{BAU:1}.
\end{proof}

\begin{remark}
The maximal monotonicity of the operator $T_{e}$ in Example~\ref{FPEX:1}\ref{BCCE:A6}
 was also  verified by Voisei and Z{\u{a}}linescu in
\cite[Example~19]{VZ} and  later a direct proof given by Bueno and
Svaiter in  \cite[Lemma~2.1]{BuSv}. Herein we have given a more
concise proof of above results.

 Bueno and Svaiter also showed  that
$T_{e}$ is not of type (D) in \cite{BuSv}. They  also showed
 that each Banach space that contains an isometric (isomorphic) copy of $c_0$
  is not of type (D) in \cite{BuSv}.
Example~\ref{FPEX:1}\ref{BCCE:A06} recaptures their result, while
Example~\ref{FPEX:1}\ref{BCCE:A2}\&\ref{BCCE:A4}
 provide a negative answer to Simons'
    \cite[Problem~22.12]{Si2}.
    \qede
\end{remark}

\begin{remark}[The continuous case] \label{rem:lat} We recall that a Banach space $X$
is a \emph{conjugate monotone space} if every continuous linear
monotone operator on  $X$ has a monotone conjugate.  In particular
this holds if every continuous linear
monotone operator on  $X$ is weakly compact.
 In consequence, a Banach lattice  $X$ contains a complemented copy of
 $\ell^1$ if and only if it admits a non (D) continuous linear
 monotone operator, on using Fact \ref{TypeDe:1}
  along with \cite[Remark 5.5]{BB} and   \cite[Examples. 5.2 and 5.3]{BB}.

 Thus, in lattices such as $c_0$, $c$ and $C[0,1]$ only
 discontinuous linear monotone operators can fail to be  of type (D). This subtlety
 escaped the current authors for fifteen years.
    \qede
\end{remark}

We now turn to a broader class of spaces:

\subsection{Applications to more general nonreflexive
spaces}\label{sec:j}

Our results below are facilitated by making use of Schauder basis
structure \cite{FabianHH2}.

\begin{definition}
We say $(e_n,e_n^*)_{n\in\NN}$ in $X\times X^*$ is  a \emph{Schauder
basis} of $X$ if for every $x\in X$
 there exists a unique sequence $(\alpha_n)_{n\in\NN}$ in
$\RR$ such that
$ x =\sum_{n\geq1}\alpha_n e_n$, where $\alpha_n
=\langle x,e^*_n\rangle$ and $\langle e_i, e^*_j\rangle
=\delta_{i,j}, \forall i,j\in\NN$.
\end{definition}
\begin{definition}
Let $(e_n,e_n^*)_{n\in\NN}$ in $X\times X^*$ be a Schauder basis of
$X$.  We say the basis is \emph{shrinking} if
$\overline{\spand\{e^*_n\mid n\in\NN\}}=X^*$.
\end{definition}

In particular, a Banach space with a shrinking basis has a separable
dual and so is an Asplund space  \cite{FabianHH2}.

\begin{fact}
\emph{(See \cite[Lemma~4.7(iii) and
Facts~4.11(ii)\&(iii)]{FabianHH2} or \cite[Lemma~6.2(iii) and
Facts~6.6(ii)\&(iii)]{FabianHH} .)}\label{SCh:1} Let
$(e_n,e_n^*)_{n\in\NN}$ in $X\times X^*$ be a Schauder basis of $X$.
Then
\begin{enumerate}
\item \label{SCh:1a}$\lim_{n}\sum_{i=1}^{n}\langle x,e^*_i\rangle e_i
=x,\quad \forall x\in X$;
\item\label{SCh:1b} $\sum_{i=1}^{n}\langle x^*,e_i\rangle e^*_i$ weak$^*$
 converges to $x^*$, written as,
$\sum_{i=1}^{n}\langle x^*,e_i\rangle e^*_i\weakstarly x^*,\quad
\forall x^*\in X^*$;
\item\label{SCh:1c} $(e_n^*, e_n)_{n\in\NN}$ in $X^*\times X^{**}$
is a Schauder basis of $\overline{\spand\{e^*_n\mid n\in\NN\}}$.
\end{enumerate}
\end{fact}

\begin{lemma}\label{LemJM:1}
Let $(e_n,e_n^*)_{n\in\NN}$ in $X\times X^*$ be a Schauder basis of
$X$. Then $e_n^*\weakstarly 0$ whenever $\liminf_{n \in
\NN}\|e_n\|>0$.
\end{lemma}
\begin{proof} Let $x\in X$.
Since $\|\langle x,e_n^*\rangle e_n\|\rightarrow0$ because of
Fact~\ref{SCh:1}\ref{SCh:1a}, and since $\liminf_{n \in
\NN}\|e_n\|>0$, we have $\langle x,e_n^*\rangle\rightarrow0$. Hence
$e_n^*\weakstarly 0$ as $n \rightarrow \infty$.
\end{proof}

The proof of Example~\ref{FPEX:B1}(i) was  inspired by that
\cite[Proposition~3.5]{BWY7}.

\begin{example}[Schauder basis]\label{FPEX:B1}
Let $(e_n,e_n^*)_{n\in\NN}$ in $X\times X^*$ be a Schauder basis of
$X$. Assume that for some $e \in X^{**}$ we have \begin{align}
\label{eq:lim}\sum_{i=1}^{n}e_i\weakstarly e\in X^{**}.\end{align}
Let $A:X\rightrightarrows X^*$ be defined by
\begin{align*}
\gra A:=\bigg\{\bigg(\sum_{n}\big(-\sum_{i>n}\langle e_i,y^*\rangle+
\sum_{i<n}\langle e_i,y^*\rangle\big)e_n,y^*\bigg)\in X\times X^*\mid y^*\in\{e\}_\bot \bigg\}.
\end{align*} Assume that $\liminf\|e_n\|>0$.
Then the following hold.
\begin{enumerate}
\item[(i)]\label{PBABA:EB2} $A$ is a maximally monotone and linear skew
operator.
\item[(ii)]\label{PBABA:EBur4} $A$ is not of type (BR).
\item[(iii)]\label{PBABA:EB3} $A$ is not of type (D).
\item[(iv)]\label{PBABA:EBu4} $A$ is not unique.

\item[(v)]\label{PBABA:EB4} Every Banach space  containing a copy of $X$ is not of type (D).
\end{enumerate}

\begin{proof}
(i): First, we show $A$ is skew.
 Let $(y,y^*)\in\gra A$. Then $\langle e, y^*\rangle=0$ and\\ $y
 =\displaystyle\sum^{\infty}_{n=1}\big(-\sum_{i>n}\langle e_i,y^*\rangle+
\sum_{i<n}\langle e_i,y^*\rangle\big)e_n$.  By the assumption that
 $\sum_{i=1}^{n}e_i\weakstarly e\in X^{**}$, we have
 \begin{align}
 s:=\sum_{i\geq 1}\langle e_i,y^*\rangle=\langle e,y^*\rangle=0.\label{EXJM:1}
 \end{align}
Thus,
\allowdisplaybreaks
\begin{align}
\langle y,y^*\rangle&=
\langle \sum_{n}\big(-\sum_{i>n}\langle e_i,y^*\rangle+
\sum_{i<n}\langle e_i,y^*\rangle\big)e_n,y^*\rangle\nonumber\\
&=\lim_{k} \langle \sum^{k}_{n=1}\big(-\sum_{i>n}\langle e_i,y^*\rangle+
\sum_{i<n}\langle e_i,y^*\rangle\big)e_n,y^*\rangle\quad \text{(by Fact~\ref{SCh:1}\ref{SCh:1a})}\nonumber\\
&=\lim_{k} \sum^{k}_{n=1}\big(-\sum_{i>n}\langle e_i,y^*\rangle+
\sum_{i<n}\langle e_i,y^*\rangle\big)\langle e_n,y^*\rangle\nonumber\\
&=-\lim_{k} \sum^{k}_{n=1}\big(\sum_{i>n}\langle e_i,y^*\rangle-
\sum_{i<n}\langle e_i,y^*\rangle\big)\langle e_n,y^*\rangle\nonumber\\
&=-\lim_{k} \sum^{k}_{n=1}\big(\sum_{i\geq n+1}\langle e_i,y^*\rangle+
\sum_{i\geq n}\langle e_i,y^*\rangle\big)\langle e_n,y^*\rangle
\quad\text{(by \eqref{EXJM:1})}\label{EXJM:b1}\\
& =-\lim_{k} \bigg(\langle e_1,y^*\rangle\sum_{i\geq 1}\langle e_i,y^*\rangle
+\langle e_2,y^*\rangle\sum_{i\geq 2}\langle e_i,y^*\rangle+\cdots
+\langle e_k,y^*\rangle\sum_{i\geq k}\langle e_i,y^*\rangle\nonumber\\
&\quad
 +\langle e_1,y^*\rangle\sum_{i\geq 2}\langle e_i,y^*\rangle
 +\langle e_2,y^*\rangle\sum_{i\geq 3}\langle e_i,y^*\rangle+\cdots+
 \langle e_k,y^*\rangle\sum_{i\geq k+1}\langle e_i,y^*\rangle\bigg)\nonumber\\
& =-\lim_{k} \bigg(s\langle e_1,y^*\rangle+(s-\langle e_1,y^*\rangle)\langle e_2,y^*\rangle+\cdots+
(s-\sum_{i=1}^{k-1}\langle e_i,y^*\rangle)\langle e_k,y^*\rangle\nonumber\\
&\quad
 +(s-\langle e_1,y^*\rangle)\langle e_1,y^*\rangle+\big(s-\sum_{i=1}^{2}
 \langle e_i,y^*\rangle\big)\langle e_2,y^*\rangle+\cdots+
(s-\sum_{i=1}^{k}\langle e_i,y^*\rangle)\langle e_k,y^*\rangle\bigg)\nonumber\\
&=-\lim_{k}\bigg(s\sum_{i=1}^{k}\langle e_i,y^*\rangle-\langle e_1,y^*\rangle
 \langle e_2,y^*\rangle-\sum_{i=1}^{2}\langle e_i,y^*\rangle\langle e_3,y^*\rangle
 -\cdots -\sum_{i=1}^{k-1}\langle e_i,y^*\rangle\langle e_k,y^*\rangle
  \nonumber\\ &\quad +
   s\sum_{i=1}^{k}\langle e_i,y^*\rangle-\sum_{i=1}^{k}\langle e_i,y^*\rangle^2
   -\langle e_1,y^*\rangle \langle e_2,y^*\rangle
   -\cdots -\sum_{i=1}^{k-1}\langle e_i,y^*\rangle\langle e_k,y^*\rangle\bigg)\nonumber\\
& =-\lim_{k}\left[2s\sum_{i=1}^{k}\langle e_i,y^*\rangle
-(\sum_{i=1}^{k}\langle e_i,y^*\rangle)^2\right]\nonumber\\
&=-(2s^2-s^2)=-s^2=0.\quad\text{(by \eqref{EXJM:1})}\nonumber
\end{align}
Hence $A$ is skew.

 To show maximality,  let $(x,x^*)\in X\times X^*$ be monotonically
related to $\gra A$. By Fact~\ref{PF:1}, we have
\begin{align}
&\langle y^*,x\rangle+\langle x^*,y\rangle=0,
 \quad \forall (y,y^*)\in\gra A.\label{EL:3}\end{align}
 By \eqref{eq:lim}, we have
 \begin{align}\langle e, e_n^*\rangle
 =\sum_{i\geq 1}\langle  e_i, e_n^*\rangle=\delta_{n,n} =1,\quad\forall n\in\NN.\label{perpro}
 \end{align}
Let $y^*:=- e_1^*+ e_n^*$ ($n\geq 2$) and $y:=-e_1-
2\sum_{i=2}^{n-1}e_i-e_n$.
By \eqref{perpro}, we have  $\langle e, y^*\rangle=0$.
Hence $y^*\in\{e\}_{\bot}$
 and $(y,y^*)\in\gra A$.
Using \eqref{EL:3},
\begin{align*}
&-\langle x,e_1^*\rangle+\langle x,e_n^*\rangle -\langle x^*,e_1\rangle
-\langle x^*,e_n\rangle-2\sum_{i=2}^{n-1}\langle x^*,e_i\rangle=0.
\end{align*}
Thus, we have
\begin{align}
&\langle x,e_n^*\rangle=\langle x,e_1^*\rangle
-\langle x^*,e_1\rangle+\langle x^*,e_n\rangle+2\sum_{i=1}^{n-1}\langle x^*,e_i\rangle.
\label{EL:4}
\end{align}
As $\sum_{i\geq 1}\langle e_i,z^*\rangle=\langle e,z^*\rangle
(\forall z^*\in X^*)$, we have $\langle x^*,e_n\rangle\rightarrow0$.

Hence, by Lemma~\ref{LemJM:1} --- since $\liminf\|e_n\|>0$ --- and
\eqref{EL:4},
\begin{align}-2\sum_{i\geq1}\langle x^*,e_i\rangle=\langle
x,e_1^*\rangle -\langle x^*,e_1\rangle.\label{EL:5}\end{align} Next
we show $-2\sum_{i\geq1}\langle x^*,e_i\rangle=\langle
x,e_1^*\rangle -\langle x^*,e_1\rangle=0$. Let
$t=\sum_{i\geq1}\langle x^*,e_i\rangle$. Then by \eqref{EL:4} and
\eqref{EL:5},
\begin{align}
x&=\sum_{n\geq1}\langle x,e^*_n\rangle e_n\nonumber\\
&=\sum_{n\geq1}
\bigg(-2\sum_{i\geq1}\langle x^*,e_i\rangle
+2\sum_{i<n}\langle x^*,e_i\rangle+
\langle x^*,e_n\rangle\bigg)e_n\nonumber\\
&=\sum_{n\geq1}\bigg(-2\sum_{i\geq n}\langle x^*,e_i\rangle
+\langle x^*,e_n\rangle\bigg)e_n\nonumber\\
&=
\sum_{n\geq1}\bigg(-\sum_{i\geq
n}\langle x^*,e_i\rangle-\sum_{i\geq n}\langle x^*,e_i\rangle
+\langle x^*,e_n\rangle\bigg)e_n\nonumber\\
&=\sum_{n\geq1}\bigg(-\sum_{i\geq n}\langle x^*,e_i\rangle
-\sum_{i\geq n+1}\langle x^*,e_i\rangle\bigg)e_n.
\label{EL:6}\end{align}

Using $(0,0)\in\gra A$,  as in the proof of \eqref{EXJM:b1}, shows
\begin{align*}
0\geq-\scal{x^*}{x}&
=\langle \sum_{n\geq1}\bigg(\sum_{i\geq n}\langle x^*,e_i\rangle
+\sum_{i\geq n+1}\langle x^*,e_i\rangle\bigg)e_n,x^*\rangle\\
 & =\lim_{k}\langle \sum_{n=1}^{k}\bigg(\sum_{i\geq n}\langle x^*,e_i\rangle
 +\sum_{i\geq n+1}\langle x^*,e_i\rangle\bigg)e_n,x^*\rangle\\
&=2t^2-t^2=t^2.
\end{align*}
Hence $t=0$.
 By \eqref{EL:6},
\begin{align*}x=\sum_{n\geq1}
\bigg(-\sum_{i>n}\langle x^*,e_i\rangle+\sum_{i<n}\langle x^*,e_i\rangle\bigg)e_n.
\end{align*}
Hence $(x,x^*)\in\gra A$. Thus, $A$ is maximally monotone.

(ii): Suppose to the contrary that $A$ is of type (BR). One checks that
$(e_1,e^*_1)\in\gra A^*$ and
 $\langle e, e^*_1\rangle=\lim_{n}\langle \sum_{i=1}^{n} e_i, e^*_1\rangle=1$. Thus,
 $(e_1,-e^*_1)\in\gra(-A^*)\cap X\times X^*$ and $-e_1^*\notin\{e\}_{\bot}$.
 Since $\overline{\ran A}\subseteq \{e\}_{\bot}$,
 $-e^*_1\notin\overline{\ran A}$. Then
 $\inf_{(a,a^*)\in\gra A}\langle e_1-a,-e^*_1-a^*\rangle=\langle e_1,-e^*_1\rangle=-1>-\infty$.
Then by Fact~\ref{BRFa:1}, $-e^*_1\in\overline{\ran A }$, which contradicts that
 $-e^*_1\notin\overline{\ran A}$. Hence $A$ is not of type (BR).

(iii): By Fact~\ref{MAS:BR1} and (ii),   $A$ is not of type (NI) and
hence $A$ is not of type (D) by Fact~\ref{PF:Su1}.
\emph{Alternative Proof:} Clearly, $(e,0)\in\gra A^*$ and thus
$e\in\ker A^*$. By the proof of (ii), $(e_1,e^*_1)\in\gra A^*$ and
 $\langle e, e^*_1\rangle=1$.
Hence $e\notin (\ran A^*)^{\bot}$. Hence $A^*$ is not monotone by
Lemma~\ref{PF:A1}. Then  Fact~\ref{TypeDe:1} shows $A$ is not of
type (D).

(iv): Apply (iii)\&(ii) and Corollary~\ref{MSCUni:1} directly.

(v): Combine (i)\&(iii)
 and Corollary~\ref{Simonco:2}.
\end{proof}
\end{example}

We shall especially exploit the lovely properties of the James
space:

\begin{definition}
The \emph{James space}, $\J$,  consists of all the sequences
$x=(x_n)_{n\in\NN}$ in $c_0$ with the finite norm
\begin{align*}
\|x\|:=\sup_{n_1<\cdots<n_k}\big((x_{n_1}-x_{n_2})^2+(x_{n_2}-x_{n_3})^2+
\cdots+(x_{n_{k-1}}-x_{n_k})^2\big)^{\tfrac{1}{2}}.
\end{align*}
\end{definition}

\begin{fact} \emph{(See \cite[page~205]{FabianHH2}
 or \cite[Claim, page~185]{FabianHH}.)} \label{fact:J} The space $\J$ is
constructed to be of codimension-one in $\J^{**}$. Indeed,
$\J^{**}=\J \oplus \spand\{e\}$ where $e:=(1,1,\ldots,1,\ldots)$ is the
constant sequence in $c(\NN) \subset \ell^\infty$. Thus, $\J$
is a separable Asplund space, equivalently $\J^*$ is separable
\cite{BorVan,FabianHH2, FabianHH}, and non-reflexive. Inter alia, the basis
$(e_n,e_n^*)_{n\in\NN}$ is a shrinking Schauder basis in $\J$ and
 $(e_n^*,e_n)_{n\in\NN}$ is a basis for $\J^*$,
 where $e_{n}=(0,\ldots,0,1,0,\ldots)$, i.e., the $n$th is $1$ and the others are $0$.
\end{fact}

\begin{corollary}[James space]\label{DSChE:1} Let  $X$ be  the James space, $\J$.
Let
$e_{n}$ be defined as in Fact~\ref{fact:J}, and let
 $A$ be defined as in Example~\ref{FPEX:B1}.
Then $A$ is a maximally monotone and skew operator that is neither
of type (BR) nor unique and so $A$ is not of type (D).
 Hence, every Banach space that contains an isomorphic copy
of $\textbf{J}$ is not of type (D).
 \end{corollary}

\allowdisplaybreaks
\begin{proof}  To apply
Example~\ref{FPEX:B1} we need only verify that \eqref{eq:lim} holds.
To see this is so, we note that
$\big(\sum_{i=1}^{n}e_i\big)_{n\in\NN}$ lies in $B_{\J^{**}}$ ---
directly from the definition of the norm in $\J$. Now  by the
Banach-Alaoglu theorem and \cite[Proposition 3.103, page~128]{FabianHH2} or \cite[Proposition 3.24,
page~72]{FabianHH}, we have the vector $e=(1,1,\ldots,1,\ldots)$ is the
unique $w^*$ limit of
$\big(\sum_{i=1}^{n}e_i\big)_{n\in\NN}$.
\end{proof}

An easier version of the same argument leads to a recovery of part
of Example \ref{FPEX:1}:

\begin{corollary}[$c_0$]\label{DSChE:2} Let  $X=c_0$. Let
$e_{n}$ be defined as in Fact~\ref{fact:J}
and $e:=(1,1,\ldots,1,\ldots)$.
Let
 $A$ be defined as in Example~\ref{FPEX:B1} (thus $A=T_{e}$ in Example~\ref{FPEX:1}\ref{BCCE:A6}).
Then $A$ is a maximally monotone and skew operator that is neither
of type (BR) nor unique and so $A$ is not of type (D). Hence, every
Banach space that contains an isomorphic copy of $c_0$ is not
of type (D).
 \end{corollary}

We finish our set of core examples by dealing with the dual space
$\J^*$.

\begin{example}[Shrinking Schauder basis]\label{FPEX:C1}
Let $(e_n,e_n^*)_{n\in\NN}$ in $X\times X^*$ be a shrinking Schauder
basis of $X$. Assume that $\sum_{i=1}^{n}e_i\weakstarly e$ for some
$e \in X^{**}$.
 Let $A:X^{*}\rightrightarrows X^{**}$ be defined by
\begin{align}\label{A:eq}
\gra A=\bigg\{(y^*,y^{**})\in X^*\times X^{**}\mid
 \sum^{k}_{n=1}\big(\sum_{i>n}\langle e_i,y^*\rangle-
\sum_{i<n}\langle e_i,y^*\rangle\big)e_n\weakstarly y^{**} \bigg\}.
\end{align}
Then  $A$ is a  maximally monotone and linear skew operator, which
is of type (BR).

In particular,  let $(e_{n})_{n\in\NN}$ and $e$
 be defined as in Fact~\ref{fact:J}. Then $A+\langle \cdot,e\rangle e$ is a maximally
 monotone operator that is neither of type (D) nor unique;
and every Banach space containing a copy of $\textbf{J}^*$ is not of
type (D).

\begin{proof}
 Again, we first show $A$ is skew.
 Let $(y^*,y^{**})\in\gra A$. Then  $$
\displaystyle\sum^{k}_{n=1}\big(\sum_{i>n}\langle e_i,y^*\rangle-
\sum_{i<n}\langle e_i,y^*\rangle\big)e_n\weakstarly y^{**}.$$  By
the assumption that
 $\sum_{i=1}^{n}e_i\weakstarly e\in X^{**}$, we have
 \begin{align}
 s:=\sum_{i\geq 1}\langle e_i,y^*\rangle=\langle e,y^*\rangle.\label{EXJMD:1}
 \end{align}
Thus,
\allowdisplaybreaks
\begin{align}
\langle y^{**}, y^*\rangle
&=\lim_{k} \langle\sum^{k}_{n=1}\big(\sum_{i>n}\langle e_i,y^*\rangle-
\sum_{i<n}\langle e_i,y^*\rangle\big) e_n,y^{*}\rangle\nonumber\\
&=\lim_{k} \sum^{k}_{n=1}\big(\sum_{i>n}\langle e_i,y^*\rangle-
\sum_{i<n}\langle e_i,y^*\rangle\big)\langle e_n,y^*\rangle\nonumber\\
&=\lim_{k} \sum^{k}_{n=1}\big(\sum_{i\geq n+1}\langle e_i,y^*\rangle+
\sum_{i\geq n}\langle e_i,y^*\rangle-s\big)\langle e_n,y^*\rangle
\quad\text{(by \eqref{EXJMD:1})}\nonumber\\
& =-s\lim_{k} \sum^{k}_{n=1}\langle e_n,y^*\rangle+\lim_{k}
 \sum^{k}_{n=1}\big(\sum_{i\geq n+1}\langle e_i,y^*\rangle+
\sum_{i\geq n}\langle e_i,y^*\rangle\big)\langle e_n,y^*\rangle\nonumber\\
&=-s^2+(2s^2-s^2)=0\quad\text{(as in the proof of
\eqref{EXJM:b1})}.\nonumber
\end{align}
Hence $A$ is skew.

Now we confirm maximality.  Let $(x^*,x^{**})\in X^*\times X^{**}$ be
monotonically related to $\gra A$. By Fact~\ref{PF:1}, we have
\begin{align}
&\langle y^*,x^{**}\rangle+\langle x^*,y^{**}\rangle=0,
 \quad \forall (y^*,y^{**})\in\gra A.\label{ELD:3}\end{align}
Fix $n\in\NN$ and set $y^*:=e_n^*$. Then
$\sum^{k}_{j=1}\big(\sum_{i>j}\langle e_i,y^*\rangle-
\sum_{i<j}\langle e_i,y^*\rangle\big)e_j=
\sum^{n-1}_{j=1}e_j-
\sum^{k}_{j=n+1}e_j$.
By the assumption that
$\sum_{i=1}^{k}e_i\weakstarly e$, we have
\begin{align*}\sum^{n-1}_{j=1}e_j-
\sum^{k}_{j=n+1}e_j\weakstarly 2\sum^{n-1}_{j=1}e_j+e_n-e.
\end{align*}
Hence
$(e_n^*,2\sum^{n-1}_{j=1}e_j+e_n-e)\in\gra A$.
Then by \eqref{ELD:3},
\begin{align*}
\langle x^{**},e_n^*\rangle +2\sum^{n-1}_{j=1}\langle x^*,e_j\rangle
+\langle x^*,e_n\rangle-\langle x^*,e\rangle=0.
\end{align*}
Since $\sum_{j\geq1}\langle x^*, e_j\rangle=\langle x^*,e\rangle$, we have
\begin{align}
\langle x^{**},e_n^*\rangle&=-2\sum^{n-1}_{j=1}\langle x^*,e_j\rangle
-\langle x^*,e_n\rangle+\langle x^*,e\rangle
=\sum_{j>n}\langle x^*,e_j\rangle-\sum_{j<n}\langle x^*,e_j\rangle.
\label{ELD:4}
\end{align}
By Fact~\ref{SCh:1}\ref{SCh:1b}\&\ref{SCh:1c},
$\sum^{k}_{n=1}\big(\sum_{j>n}\langle x^*,e_j\rangle
-\sum_{j<n}\langle x^*,e_j\rangle\big) e_n\weakstarly x^{**}$.
Hence $(x^*,x^{**})\in\gra A$. Thus, $A$ is maximally monotone.

We next show that $A$ is of type (BR). Let
$(z^*,z^{**})\in\gra(-A^*)\cap X^*\times X^{**}$. Much as in the proof above starting at
 \eqref{ELD:3}, we have $(z^*,z^{**})\in\gra A$. Thus, $\gra
(-A^*)\cap X\times X^{**} \subseteq\gra A$. Then by
Lemma~\ref{MSCBR:L1}, $A$ is of type (BR).

We turn to the particularization. By Fact \ref{fact:J},
$(e_n, e_n^*)_{n\in\NN}$ is a shrinking Schauder basis for $\textbf{J}$. By
Fact~\ref{lisum:1} since $A$ is maximal, $T=A+\langle \cdot,e
\rangle e= A+\partial \tfrac{1}{2}\langle \cdot, e\rangle^2$ is
maximally monotone. Since $A$ is skew, we have \begin{align} \langle
x^*,x^{**}\rangle=\langle x^*, e\rangle^2, \quad \forall
(x^*,x^{**})\in\gra T.\label{EAJD:05}
\end{align}
Now we claim that
\begin{align}
e\notin\overline{\ran T}.\label{EAJD:5}
\end{align}
Let $(y^*,y^{**})$ in $\gra T$. Then
\begin{align}
&\sum_{j=1}^{k}\big(2\sum_{i>j}\langle e_i,y^*\rangle+
\langle e_j,y^*\rangle\big)e_j
\nonumber\\
&=\sum_{j=1}^{k}\big(\langle y^*,e\rangle +\sum_{i>j}\langle e_i,y^*\rangle-
\sum_{i<j}\langle e_i,y^*\rangle\big)e_j
\quad\text{(by $\sum_{i\geq 1}\langle e_j, y^*\rangle=\langle e,y^*\rangle$)}\nonumber\\
&=\langle y^*,e\rangle \sum_{j=1}^{k}e_j+\sum_{j=1}^{k}\big(\sum_{i>j}\langle e_i,y^*\rangle-
\sum_{i<j}\langle e_i,y^*\rangle\big)e_j\weakstarly y^{**}.\label{EJD:1}
\end{align}
Then by \eqref{EJD:1},
\begin{align}
\lim_{k}\langle y^{**},e^*_k\rangle
&=\lim_{k}\lim_{L}\langle \sum_{j=1}^{L}\big(2\sum_{i>j}\langle e_i,y^*\rangle+
\langle e_j,y^*\rangle\big)e_j,e^*_k\rangle\nonumber\\
&=\lim_{k}\big(2\sum_{i>k}\langle e_i,y^*\rangle+
\langle e_k,y^*\rangle\big)
\nonumber\\
&=0\quad\text{(by $\sum_{k\geq 1}\langle e_k,y^*\rangle=\langle e,y^*\rangle$)}.
\end{align}
Then by Fact~\ref{fact:J}, $y^{**}\in \textbf{J}$ and hence $\ran T\subseteq \textbf{J}$.
Thus
\begin{align}\overline{\ran T}\subseteq \textbf{J}.\label{EAJD:6}\end{align}
Since $\langle e, e^*_k\rangle=1,\ \forall k\in\NN$, then by
Lemma~\ref{LemJM:1}, $e\notin \textbf{J}$. Then by \eqref{EAJD:6},
we have  \eqref{EAJD:5} holds.
Combining \eqref{EAJD:05}, \eqref{EAJD:5} and Proposition~\ref{ProJon1},
 $T=A+\langle\cdot,e\rangle e$ is neither of type (D) nor unique.

 This suffices to finish the argument.
\end{proof}
\end{example}

\begin{remark}[$\ell^1$] A simpler version  of the previous result recovers the original
result that $\ell^1$ admits Gossez type operators. \qede\end{remark}

\section{Conclusion}
We have provided various tools for the further
construction of pathological maximally monotone operators and
related Fitzpatrick functions.  In particular, we have shown ---
building on the work of Gossez, Phelps, Simons, Svaiter, Bueno and others,
and our own previous work --- that  every Banach space which
contains an isomorphic copy of either the James space $\J$ or its
dual $\J^*$,  or  $c_0$ or its dual  $\ell^1$,
 admits an operator which is not of type (D).  We observe that the type (D) property is preserved by
direct sums and subspaces.  Since every separable space is isometric to a
quotient  space of $\ell^1$ \cite[Theorem~5.1, page~237]{FabianHH2}
or \cite[Theorem~5.9, page~140]{FabianHH}, it is not preserved
by quotients.

\begin{example}[Summary] We list some of the salient spaces covered by our work:
\begin{enumerate}
\item Separable Asplund spaces: both $\J$ and $c_0$ afford examples.
\item Separable spaces whose dual is nonseparable and contain
$\ell^1$: include  $\L^1([0,1])$, $C([0,1])$ and its superspace
$L^\infty([0,1])$. \item Separable spaces whose dual is nonseparable
but does not contain a copy of  $\ell^1$: these include the James
tree space $\textbf{JT}$ \cite[page~233]{FabianHH2} or \cite[page~199]{FabianHH} as it contains
many copies of $\J$ (and of $\ell^2(\NN)$).
\end{enumerate}
One remaining potential type (D) space is  Gowers' space
\cite{Gowers} which is a non-reflexive Banach space containing
neither $c_0$, $\ell^1$ or any reflexive subspace.\qede
\end{example}

As we saw, the maximally monotone operators in our examples --- with
the exception of the Gossez operator --- that are not of type (D)
are actually  not unique. {This raises the question of how in
generality to construct maximally monotone linear relations that are
not of type (D) but that are unique.

\subsection{Graphic of classes of maximally monotone operators}

We capture much of the current state of knowledge in the following
diagram in which  the notation below is used.
\begin{align*}
``*" &\ \text{refers to  skew operators such as $T$ in
Theorem~\ref{PBABD:2},
 $T_{\alpha}$ in Example~\ref{FPEX:1},}\\
 &\quad\text{
$A$ in Example~\ref{FPEX:B1}, $A$ in Corollary~\ref{DSChE:1}, and $A$ in Corollary~\ref{DSChE:2}.} \\
``**"&\  \text{refers to the operators such as $A\& T$ in
Theorem~\ref{PBABD:2},
 $A_{\alpha}\& T_{\alpha}$ in Example~\ref{FPEX:1},}\\
 &\quad\text{
$A$ in Example~\ref{FPEX:B1}, $A$ in Corollary~\ref{DSChE:1}, $A$ in Corollary~\ref{DSChE:2},}\\
&\quad\text{ and
$A+\langle\cdot,e\rangle e$ in Example~\ref{FPEX:C1}.}\\
``***"&\  \text{denotes maximally monotone and unique operators with
non affine graphs. }
\end{align*}
We let (ANA), (FP) and  (FPV) respectively   denote the other
monotone operator classes
 ``almost negative alignment",
  ``Fitzpatrick-Phelps" and ``Fitzpatrick-Phelps-Veronas". Then by
  \cite{Si2,BorVan, Bor3,BBWY2,Si,MarSva,Si5, Yao2},
   we have the following relationships.

\begin{figure}[h]

\begin{tikzpicture}

        \node at (0,6) [rectangle,draw] {type (FPV)};
      \node at (0,3) [rectangle,draw] {***};
         \node at (3,3) [rectangle,draw]  {Gossez operator};
      \node at (-6,3) [rectangle,draw] {*};
   \node at (6,3) [rectangle,draw] {**};
       \node at (6,0) [rectangle,draw] {uniqueness};
     \node at (3,0) [rectangle,draw]  {type (ED)};
    \node at (0,0) [rectangle,draw]  {type (NI)};

    \node  at (-3,0) [rectangle,draw]  {type (FP)};
    \node  at (-2.7,1.5) [rectangle,draw]  {type (ANA)};
     \node  at (-6,1.5) [rectangle,draw]  {type (BR)};
          \node at (-6,0) [rectangle,draw]  {type (D)};
    \draw [->,thick](3,2.6)--(3,0.4);
    \draw [-,thick](2.9,1.4)--(3.1,1.6);
     \draw [->,thick](-6,3.4)--(-0.4,5.6);
        \draw [->,thick](-6,2.6)--(-6,1.9);
        \draw [-,thick](-6.07,2.15)--(-5.95, 2.35);
              \draw [->,thick](-3,0.4)--(-3,1.1);
              \draw [->,thick](-6,0.4)--(-6,1.1);
        \draw [<-,thick](-0.2,5.6)--(-5.6,0.4);
     \draw [->,thick](6,3.4)--(0.1,5.6);
      \draw [->,thick](3,3.4)--(-0.1,5.6);
   \draw [->,thick](6,2.6)--(6,0.4);
   \draw [-,thick](6.1,1.6)--(5.9,1.4);
     \draw [->,thick](0,2.6)--(0,0.4);
     \draw [->,thick](3,2.6)--(5.9,0.4);
    \draw [->,thick](4.1,0)--(4.8,0);
       \draw [<->,thick](1.1,0)--(1.9,0);
          \draw [<->,thick](-1.9,0)--(-1.1,0);
             \draw [<->,thick](-5,0)--(-4.1,0);
\end{tikzpicture}
\end{figure}
\allowdisplaybreaks

The following four questions are left open.
\begin{enumerate}
\item Is every maximally monotone operator  necessarily of type (FPV)?
\item Is every maximally monotone operator necessarily of type (ANA)?
\item Is every maximally monotone linear relation necessarily of type (ANA)?
\item Is every maximally monotone operator of type (BR) necessarily of type (ANA)?
\end{enumerate}

The first of these is especially important, being closely related to
the sum theorem in general Banach space (see \cite{Si2,BorVan,Bor3,Yao3}).

 \vfill

\paragraph{Acknowledgments.} Heinz Bauschke was partially supported by the
Natural Sciences and Engineering Research Council of Canada and by
the Canada Research Chair Program. Jonathan  Borwein was partially
supported by the Australian Research  Council.
Xianfu Wang was
partially supported by the Natural Sciences and Engineering Research
Council of Canada.


\newpage

\begin{thebibliography}{99}
\bibitem{BAU:1}
H.H.\ Bauschke, \emph{Projection Algorithms and Monotone Operators}, PhD thesis, Simon
Fraser University, Department of Mathematics, Burnaby, British Columbia V5A 1S6,
Canada, August 1996; Available at http://www.cecm.sfu.ca/preprints/1996pp.html.


\bibitem{BB}
H.H.\ Bauschke and J.M.\ Borwein, ``Maximal monotonicity of dense
type, local maximal monotonicity, and monotonicity of the conjugate
are all the same for continuousi linear operators'', \emph{Pacific
Journal of Mathematics}, vol.~189, pp.~1--20, 1999.

\bibitem{BWY7}
H.H.\ Bauschke, X.\ Wang, and L.\ Yao,
``Examples of discontinuous
maximal monotone linear operators
and the solution to a recent problem posed by B.F.~Svaiter'',
\emph{Journal of Mathematical Analysis and Applications},
 vol.~370, pp. 224-241,
 2010.

\bibitem{BBWY:1}
H.H.\ Bauschke, J.M.\ Borwein,  X.\ Wang and L.\ Yao,
``For maximally monotone linear relations,
dense type, negative-infimum type, and Fitzpatrick-Phelps type
all coincide with monotonicity of the adjoint'', submitted;
\texttt{http://arxiv.org/abs/1103.6239v1}, March 2011.
\bibitem{BBWY2}
 H.H.\ Bauschke, J.M.\ Borwein, X.\ Wang, and L.\ Yao,
 ``Every maximally monotone operator of
Fitzpatrick-Phelps type is actually of dense type'', submitted;
\texttt{http://arxiv.org/abs/1104.0750v1}, April 2011.

\bibitem{BC2011}
H.H.\ Bauschke and P.L.\ Combettes,
\emph{Convex Analysis and Monotone Operator Theory in Hilbert Spaces},
Springer-Verlag, 2011.





\bibitem{Bor1}
J.M.\ Borwein,
``Maximal monotonicity via convex analysis'',
\emph{Journal of Convex Analysis}, vol.~13, pp.~561--586, 2006.

\bibitem{Bor2}J.M.\ Borwein, ``Maximality of sums of two maximal monotone operators in general
Banach space'',
\emph{Proceedings of the
 American Mathematical Society}, vol.~135, pp.~3917--3924, 2007.

\bibitem{Bor3}J.M.\ Borwein, ``Fifty years of maximal monotonicity'',
\emph{Optimization Letters}, vol.~4, pp.~473--490, 2010.




\bibitem{Borwein2}
 J.M.\  Borwein, ``A note on $\varepsilon$-subgradients and maximal monotonicity'',
  \emph{Pacific Journal of Mathematics},
 vol.~103, pp.~307--314, 1982.



\bibitem{BorVan}
J.M.\ Borwein and J.D.\ Vanderwerff,
\emph{Convex Functions},
Cambridge University Press, 2010.




\bibitem{BuSv}
 O.\ Bueno and  B.F.\ Svaiter,
``A non-type (D) operator in $c_0$'',\\
\texttt{http://arxiv.org/abs/1103.2349v1}, March 2011.

\bibitem{BurIus}
R.S.\ Burachik and A.N.\ Iusem,
\emph{Set-Valued Mappings and Enlargements of Monotone Operators},
Springer-Verlag, 2008.

\bibitem{Cross}
R.\ Cross,
\emph{Multivalued Linear Operators},
Marcel Dekker, 1998.


\bibitem{FabianHH}
M.\ Fabian, P.\ Habala, P.\  H\'{a}jek, V.\ Montesinos Santaluc\'{i}a, J.\ Pelant and V. Zizler,
 \emph{Functional Analysis and Infinite-Dimensional Geometry}, CMS/Springer-Verlag, 2001.

\bibitem{FabianHH2}
M.\ Fabian, P.\ Habala, P.\  H\'{a}jek, V.\ Montesinos  and V.\
Zizler,
 \emph{Banach Space Theory}, CMS/ Springer-Verlag, 2010.


\bibitem{Fitz88}
S.\ Fitzpatrick, ``Representing monotone operators by convex
functions'', in  \emph{Workshop/Miniconference on Functional
Analysis and Optimization (Canberra 1988)}, Proceedings of the
Centre for Mathematical Analysis, Australian National University,
vol.~20, Canberra, Australia, pp.~59--65, 1988.

\bibitem{FP}
S.\ Fitzpatrick
and R.R.\  Phelps, ``Bounded approximants to monotone operators on Banach spaces'',
 \emph{ Annales de l'Institut Henri Poincar\'{e}. Analyse Non Lin\'{e}aire},  vol.~9, pp.~573--595, 1992.

\bibitem{Gossez3}
J.-P.\ Gossez, ``Op\'{e}rateurs monotones non lin\'{e}aires dans les espaces de Banach non r\'{e}flexifs'',
\emph{Journal of Mathematical Analysis and Applications}, vol.~34, pp.~371--395,  1971.


\bibitem{Gossez1}
J.-P.\ Gossez, ``On the range of a coercive maximal monotone operator in a nonreflexive
Banach space'',
\emph{Proceedings of the
 American Mathematical Society}, vol.~35, pp.~88-–92,  1972.


\bibitem{Gowers} W. T.\  Gowers, ``A Banach space not
 containing $c_0$, $l_1$ or a reflexive subspace",
\emph{Transactions of the American Mathematical Society},
 vol.~.344,  pp.~407–-420, 1994.



 \bibitem{MSV}M.\ Marques Alves and B.F.\ Svaiter,
``A new proof for maximal monotonicity of subdifferential operators'',
\emph{Journal of Convex Analysis}, vol.~15, pp.~345--348, 2008.

\bibitem{MarSva3}
M.\  Marques Alves and B.F.\ Svaiter,
``Br{\o}ndsted-Rockafellar property and maximality of
monotone operators representable by convex
functions in non-reflexive Banach spaces'',
\emph{Journal of Convex Analysis},
vol.~15, pp.~693–-706, 2008.



\bibitem{MarSva2}
M.\  Marques Alves and B.F.\ Svaiter,
``Maximal monotone operators with
a unique extension to the bidual'',
\emph{Journal of Convex Analysis},
vol.~16, pp.~409–-421, 2009


\bibitem{MarSva}
M.\  Marques Alves and B.F.\ Svaiter,
``On Gossez type (D)
maximal monotone operators'',
\emph{Journal of Convex Analysis},
vol.~17, pp.~1077--1088, 2010.


\bibitem{Megg}
R.E.\  Megginson,
\emph{An Introduction to Banach Space Theory},
Springer-Verlag, 1998.


\bibitem{ph}
R.R.\ Phelps,
\emph{Convex Functions, Monotone Operators and
Differentiability},
2nd Edition, Springer-Verlag, 1993.

\bibitem{ph2}
R.R.\ Phelps,
``Lectures on maximal monotone operators'',
 \emph{Extracta Mathematicae}, vol.~12, pp.~193--230, 1997;\\
\texttt{http://arxiv.org/abs/math/9302209v1}, February 1993.

\bibitem{PheSim}
R.R.\ Phelps and S.\ Simons, ``Unbounded linear monotone
operators on nonreflexive Banach spaces'',
\emph{Journal of Convex Analysis}, vol.~5, pp.~303--328, 1998.


\bibitem{Rock702}
R.T.\ Rockafellar,
``On the maximal monotonicity of subdifferential mappings'',
\emph{Pacific Journal of Mathematics},
vol.~33, pp.~209--216, 1970.



\bibitem{RockWets}
R.T.\ Rockafellar and R.J-B Wets,
\emph{Variational Analysis}, 3rd Printing,
Springer-Verlag, 2009.


\bibitem{SiNI}
S.\  Simons,
``The range of a monotone operator'',
\emph{Journal of Mathematical Analysis and Applications}, vol.~199, pp.~176--201,  1996.



\bibitem{Si}
S.\  Simons,
\emph{Minimax and Monotonicity},
Springer-Verlag, 1998.

\bibitem{Si4}S.\  Simons,
``Five kinds of maximal monotonicity'',
\emph{Set-Valued and Variational Analysis},
vol.~9, pp.~391--409, 2001.

\bibitem{Si2}
S.\ Simons, \emph{From Hahn-Banach to Monotonicity},
Springer-Verlag, 2008.


\bibitem{Si5}
S.\ Simons,
``Banach SSD Spaces and
classes of monotone sets'',
\emph{Journal of Convex Analysis}, vol.~18, pp.~227--258,  2011.

\bibitem{Si6} S.\ Simons, ``Maximal monotone multifunctions of
Br{\o}ndsted-Rockafellar type", \emph{Set-Valued
Analysis}, vol. 7 pp.~255-–294, 1999.

\bibitem{SiZ}
 S.\ Simons and C.\  Z{\v{a}}linescu, ``Fenchel duality, Fitzpatrick functions
  and maximal monotonicity'',
  \emph{Journal of Nonlinear and Convex Analysis}, vol.~ 6, pp. 1--22, 2005.



\bibitem{VZ}
M.D.\ Voisei and C.\ Z{\u{a}}linescu,
``Linear monotone subspaces of locally
convex spaces'',
\emph{Set-Valued and Variational Analysis},
vol.~18, pp.~29--55, 2010.

\bibitem{Yao3}
L.\ Yao,  ``The sum of a maximal monotone operator  of type (FPV) and a maximal monotone operator
with full domain is maximally monotone'', to appear \emph{Nonlinear Analysis}.


\bibitem{Yao2}
L.\ Yao,  ``The sum of a maximally monotone
 linear relation and the subdifferential of a proper lower semicontinuous
  convex function is maximally monotone'', to appear \emph{Set-Valued and Variational Analysis}.

\bibitem{Zalinescu}
{C.\ Z\u{a}linescu},
\emph{Convex Analysis in General Vector Spaces}, World Scientific
Publishing, 2002.



\end{thebibliography}
\end{document}